\documentclass{amsart}
\usepackage[utf8]{inputenc}

\usepackage{mathptmx}
\usepackage{paralist}
\usepackage{graphics}
\usepackage{thmtools}
\usepackage{wasysym}
\usepackage[T1]{fontenc}    

\usepackage{amsthm}
\usepackage{amsbsy,amsmath,amssymb,amscd,amsfonts,float}
\usepackage[pagebackref=false]{hyperref}
\usepackage[nameinlink,capitalize,noabbrev]{cleveref}

\usepackage{graphicx,float,latexsym,color}
\usepackage[font={scriptsize,it}]{caption}
\usepackage{subcaption}
\usepackage[export]{adjustbox}

\usepackage{makecell}

\usepackage[dvipsnames]{xcolor}

\newtheorem{theorem}{Theorem}
\newtheorem*{theorem*}{Theorem}

\newtheorem{proposition}{Proposition}
\newtheorem{conjecture}{Conjecture}
\newtheorem{corollary}{Corollary}

\theoremstyle{remark}
\newtheorem{remark}{Remark}
\theoremstyle{definition}
\newtheorem{definition}{Definition}

\hypersetup{
    pdftoolbar=true,        
    pdfmenubar=true,        
    pdffitwindow=false,     
    pdfstartview={FitH},    
    colorlinks=true,       
    linkcolor=OliveGreen,          
    citecolor=blue,        
    filecolor=black,      
    urlcolor=red           
}

\usepackage{lineno}

\usepackage{comment}
\usepackage{microtype}
\usepackage{footnote}

\newcommand{\E}{\mathcal{E}}

\newcommand{\C}{\mathcal{C}}

\renewcommand{\L}{\mathcal{L}}
\renewcommand{\P}{\mathcal{P}}

\newcommand{\Y}{\mathcal{Y}}
\renewcommand{\H}{\mathcal{H}}

\usepackage[compact]{titlesec}
\titleformat{name=\section}{}{\thetitle.}{1em}{\centering\scshape}
\titleformat{name=\subsection}[runin]{}{\thetitle.}{0em}{\bfseries}
\titlespacing{\section}{0pc}{1.5ex plus .1ex minus .2ex}{0pc}

\titlespacing{\subsection}{0pc}{0.5pc}{0.5em}

\crefname{conjecture}{Conjecture}{Conjectures}
\crefname{definition}{Definition}{Definitions}

\newcommand{\rulesep}{\unskip\ \vrule\ }
\usepackage{esvect}

\title{Triads of Conics Associated with a Triangle\vspace{-1em}}
\author[R. Garcia]{Ronaldo Garcia}
\thanks{R. Garcia, Federal Univ. of Goiás, Brazil. \texttt{ragarcia@ufg.br}}
\author[L. Gheorghe]{Liliana G. Gheorghe}
\thanks{L. Gheorghe, Federal Univ. of Pernambuco, Brazil. \texttt{liliana@dmat.ufpe.br}}
\author[P. Moses]{Peter Moses}
\thanks{P. Moses, Moparmatic Inc., Worcestershire, England, \texttt{moparmatic@gmail.com}}
\author[D. Reznik]{Dan Reznik}
\thanks{D. Reznik$^*$, Data Science Consulting Ltd., Rio de Janeiro, Brazil. \texttt{dreznik@gmail.com}}

\begin{document}

\maketitle
{\vspace{-.75cm}
\centering\small Dedicated to Paul Yiu \par}

\vspace{-.25cm}
\begin{abstract}
We revisit constructions based on triads of conics with foci at pairs of vertices of a reference triangle. We find that their 6 vertices lie on well-known conics, whose type we analyze. We give conditions for these to be circles and/or degenerate. In the latter case, we study the locus of their center. 
\end{abstract}


\section{Introduction}
Paraphrasing a passage in \cite{sharp2015-artzt}, ``new tools of interactive geometry enable the discovery of properties in a way mathematicians in the past could only have dreamed about''. Aided by interactive simulation (mostly Mathematica and GeoGebra), and inspired by a construction by Paul Yiu \cite[Sec. 12.4, p. 148]{yiu2001-intro}, we tour curious dynamic phenomena manifested by triads of ellipses (or hyperbolas) naturally associated with a triangle. Namely, we attach their foci to a pair of vertices and impose that  the conic pass through either (i) the remaining vertex, or (ii), some chosen point $P$. We call these ``V-'' or ``P''-conics, respectively, see \cref{fig:intro}. Some of our main results include:

\begin{figure}
    \centering
    \includegraphics[trim=60 120 60 110,clip,width=\textwidth]{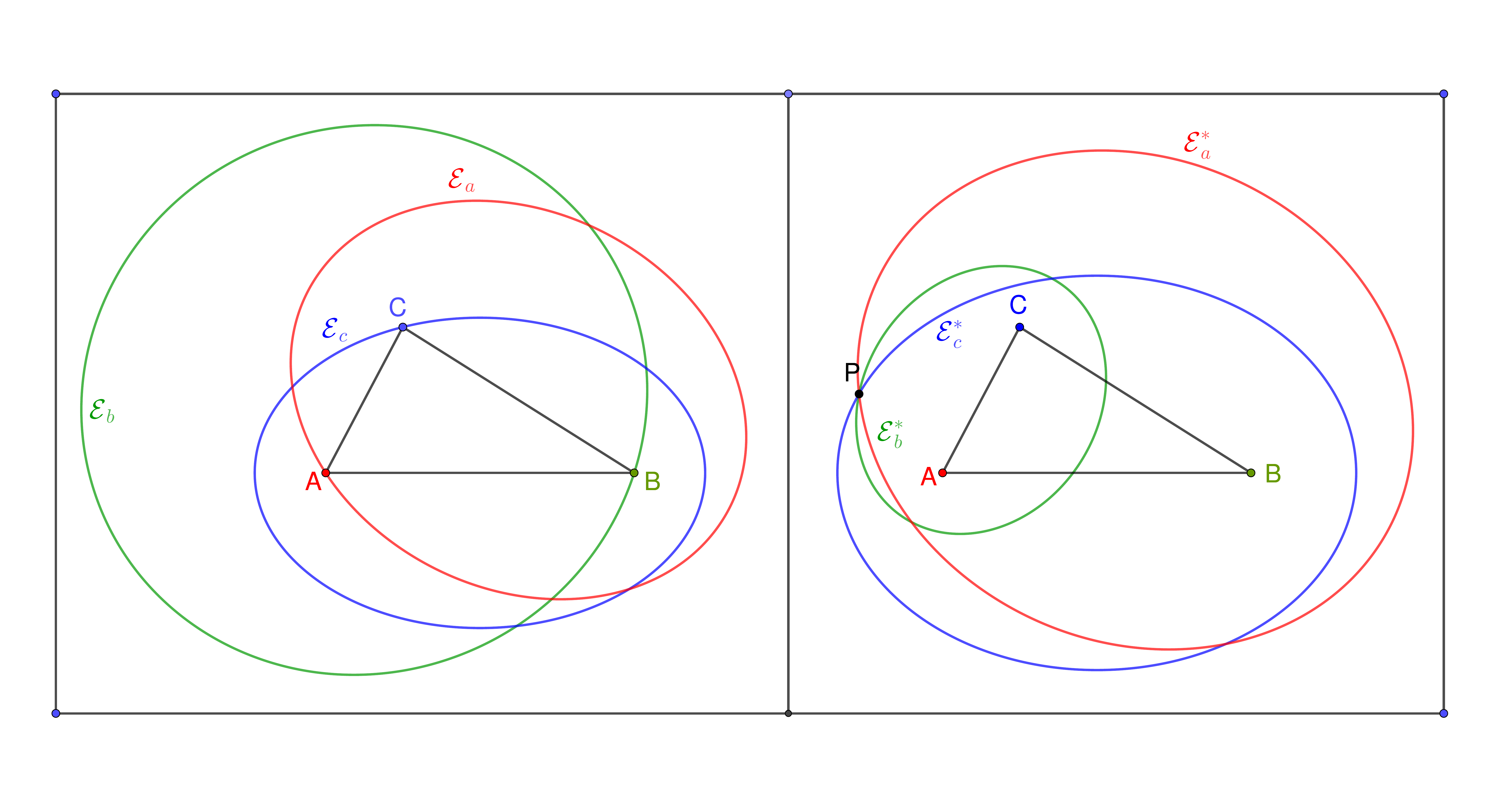}
    \caption{\textbf{Left:} A $\triangle ABC$, and a V-triad of ellipses passing through a vertex and with foci on the remaining pair. \textbf{Right:} in the P-triad case, ellipses still have foci on pairs of vertices but now pass through a given point $P$.}
    \label{fig:intro}
\end{figure}

\begin{compactitem}
    \item The 6 vertices of V-ellipses always lie on a conic; this conic
    is degenerate iff the reference triangle is a right triangle.
    \item The conic passing through the 6 $P$-ellipse vertices is degenerate iff $P$ lies on the circumcircle.
    \item The locus of the center of the 6-point conic over the degenerate family is a quartic in the V-ellipse case, and the union of three arcs of ellipses in the P-ellipse case; we derive expressions for them.
    \item We specify the regions such that various 6-point conics are of a given type (hyperbola, ellipse, parabola, or degenerate).
    \item We derive conditions such that various 6-point conics are a circle.
    \item We derive conditions (and loci) under which the co-vertices of conic triad lie on a conic.
\end{compactitem}

Some of the above are done for the case of hyperbola triads as well. Most of our results have been obtained through experimentation with dynamic geometry software first, and later
confirmed geometrically and/or algebraically.
See \cite{garcia2020-ellipses} for details.

Some long, symbolic proofs are omitted, with some expressions appearing in \cref{app:barys}. Throughout the paper we will be using $X_k$ notation for triangle centers, after \cite{etc}.

\subsection*{Related Work}

We have been inspired by the idea of erecting identical geometrical objects to the sides of a triangle, e.g.,  \cite{ cerin2002-flanks,dosa2007-ext, lamoen2004-flank, moses2021-hex}. Triads of ``Artzt'' parabolas, conceived in the XIX century, have been revisited in \cite{dergiades2010-artzt,klaoudatos2011-artzt,thebault1957-parabolas}. In \cite{sharp2015-artzt}, new properties of Artzt parabolas are detected via dynamic geometry software. Properties of conic triads with a shared focus are studied in  \cite{bogdanov2012-ellipses}. A 6-point conic passing through the tangency point of the excircles (which turns out to coincide with the vertices of V-ellipses) is described in \cite{bradley-conic,yiu1999-clawson}. A Construction of 3 ``Soddy'' hyperbolas (called here V-hyperbolas) with foci on vertices appears in \cite[Sec. 12.4, p. 148]{yiu2001-intro}. Properties of a triad of circles tangent to the nine-point circle are studied in \cite{odehnal2014-triads}.

\subsection*{Article organization}

Properties of triads of V-ellipses, P-ellipses, V-hyperbolas, and P-hyperbolas, are covered in \cref{sec:v-ell,sec:p-ell,sec:v-hyp,sec:p-hyp}, respectively. In 
last section we pose to the reader a few open questions. The last section  contains some long-form symbolic expressions for a construction appearing in \cref{sec:v-ell}.

\section{A triad of V-ellipses}
\label{sec:v-ell}
Referring to \cref{fig:intro}:

\begin{definition}[V-ellipses]
Given a triangle $\triangle ABC$, a triad of V-ellipses $\E_a,\E_b,\E_c$ have foci on $(B,C)$, $(C,A)$, $(A,B)$ and pass through $A$, $B$, and $C$, respectively.
\end{definition}

\begin{proposition} The V-ellipses
$\E_a,\E_b,\E_c$ are centered at the midpoints
of $\triangle{ABC}$'s sides. Their vertices\footnote{These refer to the intersection of a conic with the focal axis.} are the (external) tangency points of the excircles with triangle's  sidelines and lie on a conic, $\Y$.
\label{prop:yiu3}
\end{proposition}

\begin{proof}
Let $a,b,c$ be the sidelengths of $\triangle ABC$.
Let $({I}_a),({I}_b),({I}_c)$ the escribed circles and let 
 $A_1$, $A_2$, $B_1$, $B_2$, $C_1$, $C_2$  their (external) tangency points with the lines $BC$, $CA$, $AB$,
 as shown in \cref{fig:x20}.
 We shall prove that these points are the  intersection  of the V-ellipses with 
their focal axis  $BC$, $CA$, $AB$, hence their vertices.
Elementary properties of tangents from a point to a circle yield:
\begin{align}
AC_2=&AB_1=BA_2=BC_1=CA_1=CB_2=p, \nonumber \\
BA_1=&CA_2=p-a,\;\;\;
AB_2=CB_1=p-b,\;\;\;
AC_1=BC_2=p-c  \label{eq:tg1} 
\end{align}
where $p=(a+b+c)/2$ is the semi-perimeter. Hence:
\[A_1 A_2=A_1B+BC+CA_2=(p-a)+a+(p-a)=
2p- a=b+c.\]
Since $A_1 B= A_2C,$ and since points $A_1$, $A_2$, $B$ and $C$ are collinear, the former are precisely the two vertices of $\E_a$. Furthermore, the segments 
 $BC$ and $A_1 A_2$ share their midpoint,   the center of $\E_a$. The proof for $\E_b$ and $\E_c$
is similar. 

 In order to prove that their six  vertices
 are on a conic,
 by Carnot's Theorem, it is enough to check that
\begin{equation}
\frac{AC_1}{BC_1}\cdot\frac{AC_2}{BC_2}\cdot
\frac{B A_1}{CA_1}\cdot \frac{BA_2}{CA_2}\cdot
\frac{CB_1}{AB_1}\cdot \frac{CB_2}{AB_2}=1
\label{eq:carnot}
\end{equation}

This claim is obtained by substituting \cref{eq:tg1} into \cref{eq:carnot}.
\end{proof}

\begin{remark} The fact that a conic passes through the six external tangency points with the excircles was discovered by Paul Yiu \cite{yiu1999-clawson}. In \cite{etc} its center is labeled $X_{478}$.
\end{remark}

It can be shown that the Yiu conic $\mathcal{Y}$ can never be a circle except when $\triangle ABC$ is an equilateral.

\begin{proposition}
Each V-ellipse $\E_a,\E_b,\E_c$ is respectively tangent at $A,B,C$ to the sides of the excentral triangle.
\end{proposition}

\begin{proof} Referring to \cref{fig:x20},
since $I_a,I_b,I_c$ are the centers of the escribed circles, the lines $I_b I_c$, 
$I_c I_a$, $I_a I_b$ are the external bisectors
of $\angle{BAC}$, $\angle{ACB}$, and $\angle{BCA}$; thus
 $A I_a$, $BI_b$ $CI_c$ are altitudes in $\triangle{I_a I_b I_c}$
as well as (internal) bisectors of $\triangle{ABC}.$
By the optic propriety of conics, lines $I_b I_c$, $I_c I_a$, $I_a I_b$ are also the tangents in $A$, $B$, $C$ to the ellipses
$\E_a$, $\E_b$, $\E_c$.
\end{proof}

Referring to \cref{fig:x20}, let $(A',A'')$, $(B',B'')$, and $(C',C'')$ denote the pairwise intersections between $(\E_b,\E_c)$, $(\E_a,\E_c)$, and $(\E_a,\E_b)$, respectively. 

\begin{proposition}
The lines through $A',A''$, $B',B''$, $C',C''$ pass through the 3 excenters $I_a,I_b,I_c$, respectively, and concur at  the de Longchamps' point $X_{20}$.
\end{proposition}

\begin{proof}
 It can be shown that the $\E_a$ is given by the following implicit equation in barycentric coordinates $[x,y,z]$:
\begin{align*}
\E_a:\;\;&4 c (b + c) x y - (a - b - c) (a + b + c) y^2 + 4 b (b + c) x z \;+\\
&+\; 2 (a^2 + b^2 + 2 b c + c^2) y z - (a - b - c) (a + b + c) z^2 = 0
\end{align*}
$\E_b,\E_c$ can be obtained cyclically on $a,b,c$. The barycentrics for the vertices of $\E_a$ are $A_1=[0, a + b + c, a - b - c]$ and $A_2=[0, a - b - c, a + b + c]$.
Let $S$ be twice the area of $\triangle ABC$. The two real intersections $A',A''$ between $\E_b,\E_c$ are given by:
{\small
\begin{align*}
A'=&\left[(a - b - c)(a + b - c)(a - b + c)(3a^2 + 2 a b - b^2 + 2 a c + 2 b c - c^2)\right. +\\
&+4 (-2 a^3 - a^2 b - b^3 - a^2 c + b^2 c + b c^2 - c^3) S,\\
&(a - b - c) (a - b + c) (a + b + c) (a^2 - 2 a b - 3 b^2 + 2 a c + 2 b c + c^2) +\\ 
&+4 (a^3 + a b^2 + 2 b^3 + a^2 c - b^2 c - a c^2 - c^3) S,\\
&(a - b - c) (a + b - c) (a + b + c) (a^2 + 2 a b + b^2 - 2 a c + 2 b c - 3 c^2) +\\
&\left. +4 (a^3 + a^2 b - a b^2 - b^3 + a c^2 - b c^2 + c^3) S \right]
\end{align*}
}
and $A''$ is obtained as above but with $S \to -S$. The intersections $B',B''$ and $C',C''$ are obtained cyclically. The line $A'A''$ is then given by:
\[-(b - c)(a + b + c)^2 x  -(a + b - c)^2 (a + c) y + (a + b) (a - b + c)^2 z = 0\]

It can be shown this line passes through excenter $I_a$. The other lines can be obtained cyclically. It can also be shown these meet at $X_{20}$, whose first barycentric coordinate is given by \cite{etc}: $[-3 a^4 + 2 a^2(b^2 + c^2) + (b^2 - c^2)^2]$, with the other two obtained cyclically.
\end{proof}

Referring to \cref{fig:locus-v-triad}:

\begin{proposition}
When $\triangle ABC$ is a right triangle, the V-ellipses pass through the reflection of the orthocenter on the circumcenter, the de Longchamps point $X_{20}$.
\end{proposition}

\begin{proof}
Let $C$ denote the right-angle vertex of $\triangle ABC$, and $C'$ its reflection about the circumcenter $X_3$. We shall prove that each V-ellipse passes through $C'$. Due to central symmetry, this is trivially true for $\E_c$. Consider $\E_a$: since its foci are $B,C$ and it passes through $A$, its major axis has length $|AC|+|AB|$. Since $ACBC'$ is a rectangle, $|AC|=|BC'|$, and $|BC|=|C'A|$. Hence $ |C'B|+|CC'|=|AC|+|AB|$, which ensures that $C'\in \E_a$. Similarly $C'\in \E_b.$
\end{proof}

\begin{figure}
    \centering
    \includegraphics[width=.8\textwidth]{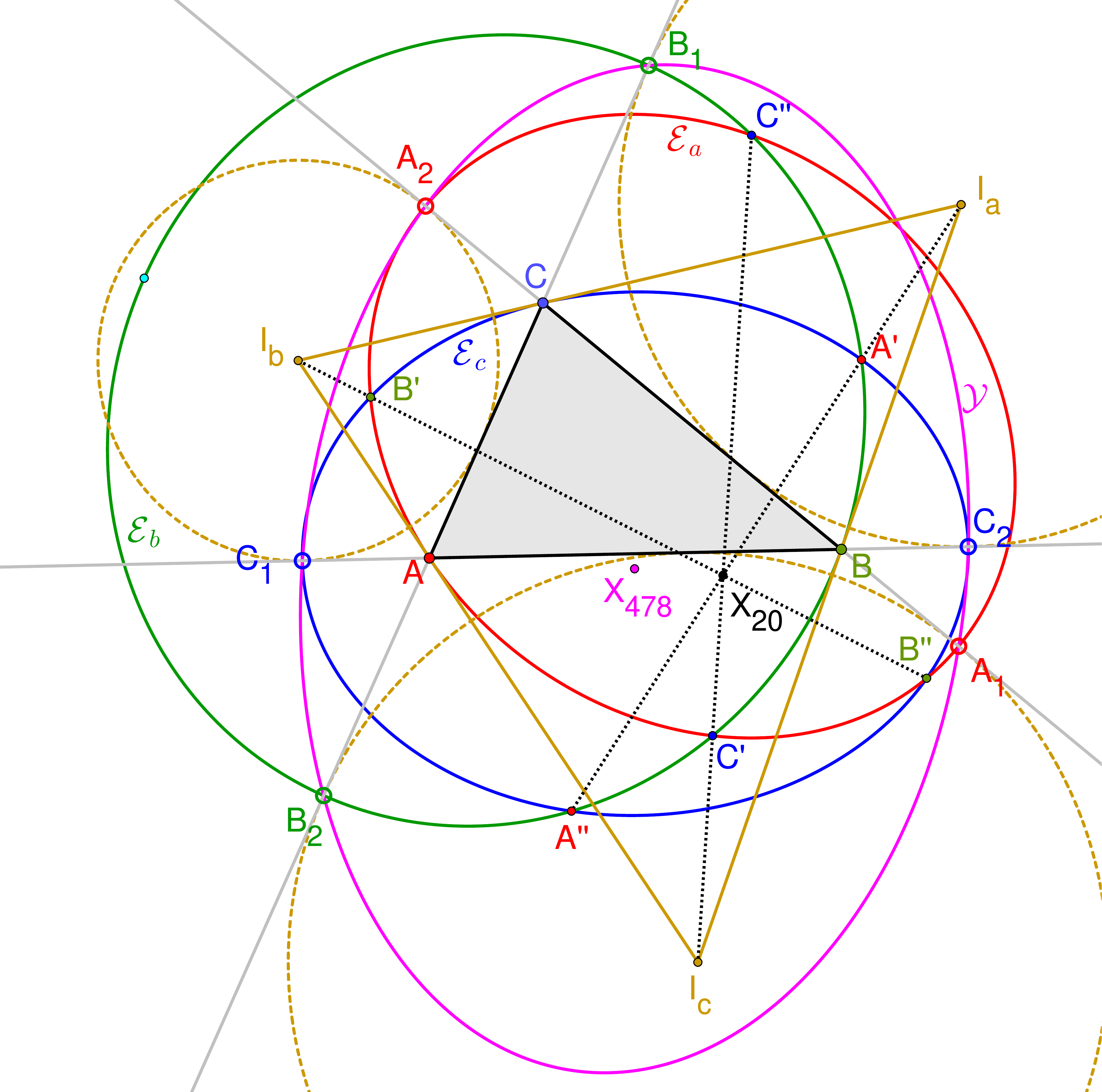}
    \caption{Properties of a V-ellipses $\E_a,\E_b,\E_c$ (red, green, blue) with respect to a $\triangle ABC$ (black). (i) Its vertices are at the tangency points of the excircles (dashed gold) with the sidelines; hence they lie on the Yiu conic (magenta) \cite{yiu1999-clawson}. (ii) Each ellipse is tangent at $A,B,C$ to a side of the excentral triangle $\triangle{I_a I_b I_c}$. (iii) The 3 chords $A'A''$, $B'B''$, and $C'C''$ between the intersections of $(\E_b,\E_c)$, $(\E_c,\E_a)$,
    $(\E_a,\E_b)$
    pass through $I_a,I_b,I_c$, and concur at $X_{20}$.}
    \label{fig:x20}
\end{figure}

\begin{figure}
    \centering
    \includegraphics[trim=-10 50 180 400,clip,width=.8\textwidth,frame ]{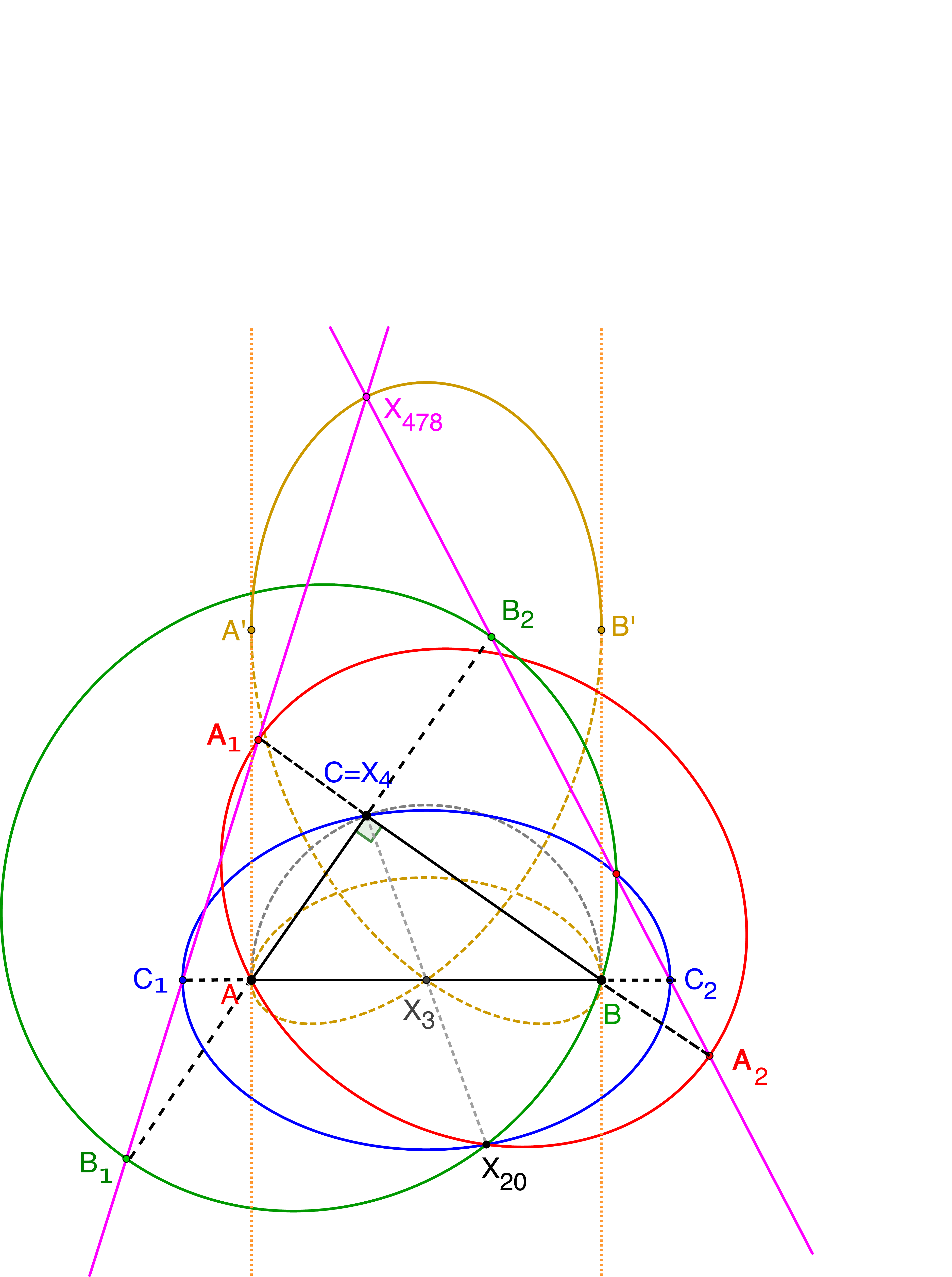}
    \caption{If $\triangle ABC$ is a right triangle, the Yiu conic $\Y$ (magenta) is degenerate, and the V-ellipses intersect at $X_{20}$. Furthermore, over all $C$ on a semicircle with $AB$ as a diameter, the locus of the center $X_{478}$ of $\Y$ is an arc (solid gold) of a quartic (dashed gold). The lines through $A,B$ perpendicular to $AB$ (dotted gold) are tangent to the locus at its endpoints $A'$, $B'$, and $|AA'|=|BB'|=|AB|$.}
    \label{fig:locus-v-triad}
\end{figure}

\subsection*{Degenerate six-point conic:} Still referring to \cref{fig:locus-v-triad}:

\begin{proposition} $\Y$ is degenerate
iff $\triangle ABC$ is a right triangle. 
\end{proposition}

\begin{proof}
  Let $A_1,A_2,B_1,B_2,C_1,C_2$  the intersection points of the ellipses, with 
the lines $BC$, $CA$, $AB$, as in \cref{fig:locus-v-triad}.
We shall prove  that
$A_1,B_1,C_1$ are collinear iff
 $\triangle{ABC}$ is right-angled.
To do so, by Menelaus' theorem, we need to check that
{\small\begin{equation}
\frac{A_1 C}{A_1 B}\cdot \frac{C_1 B}{C_1 A}
\cdot \frac{B_1 A}{B_1 C}=1.
\label{eq:menelau}
\end{equation}}
Let   $x=CA_1=BA_2,\;\;y=AB_1=CB_2,\;\; z=AC_1=BC_2$.

Since the $V-$ ellipses pass through one of triangle's vertices and have their foci into the other two,
$a+2x=b+c,\;\;\;b+2y=a+c,\;\;\;c+2z=a+b,$
 hence \[x=p-a,\;\; y=p-b,\;\; z=p-c,\] 
where $p=\frac{a+b+c}{2}$ is the semi-perimeter. 
Substituting this into \cref{eq:menelau}, we obtain:
\[\frac{x}{p}\cdot \frac{p}{z}\cdot \frac{y}{p}=1      \]
hence $(p-a)\cdot (p-b)=p \cdot (p-c)$, which is
equivalent to 
$c^2=a^2+b^2$. The result follows by Pythagoras' theorem.
\end{proof}

\noindent Assume, without loss of generality, that $A=(1/2,0)$ and $B=(-1/2,0)$.
\begin{proposition}
Over $C$ on the semicircle whose diameter is $AB$, $y>0$, the locus of the center of the degenerate $\Y$ is the arc of a quartic given by:
\[4(x^2 + y^2)^2 - 8y^3 - x^2 + 2y^2=0,\;\;\;y>1\]
The semicircle with $y<0$ produces a locus which is symmetric about the $x$-axis.

\end{proposition}


\begin{proof}
The claim was obtained via manipulation and simplification with a Computer Algebra System (CAS).
\end{proof}

Referring to \cref{fig:vells-zones}:

\begin{proposition}
With $A$, $B$ fixed, the Yiu conic $\Y$ of $\triangle ABC$ is (i) degenerate if $C$ lies on the union of the circumcircle with the two lines tangent to it at $A$ and $B$; (ii) a parabola if $C$ lies on a curve whose barycentrics satisfy the following degree-8 implicit equation:
\begin{align*}
& a^8 + b^8+c^8 - 2(a^4 b^4 +a^4 c^4 + b^4 c^4) +\\
& 4 a b c ( a^5 + b^5 + c^5 - a^4 b - a b^4 -  a^4 c - a c^4 -  b c^4  - b^4 c + 
 a^3 b c +  a b^3 c +  a b c^3)= 0
\end{align*}

\end{proposition}

\begin{proof}
The claim was obtained via manipulation and simplification with a Computer Algebra System (CAS).
\end{proof}

\begin{figure}
    \centering
    \includegraphics[width=.7\textwidth,frame]{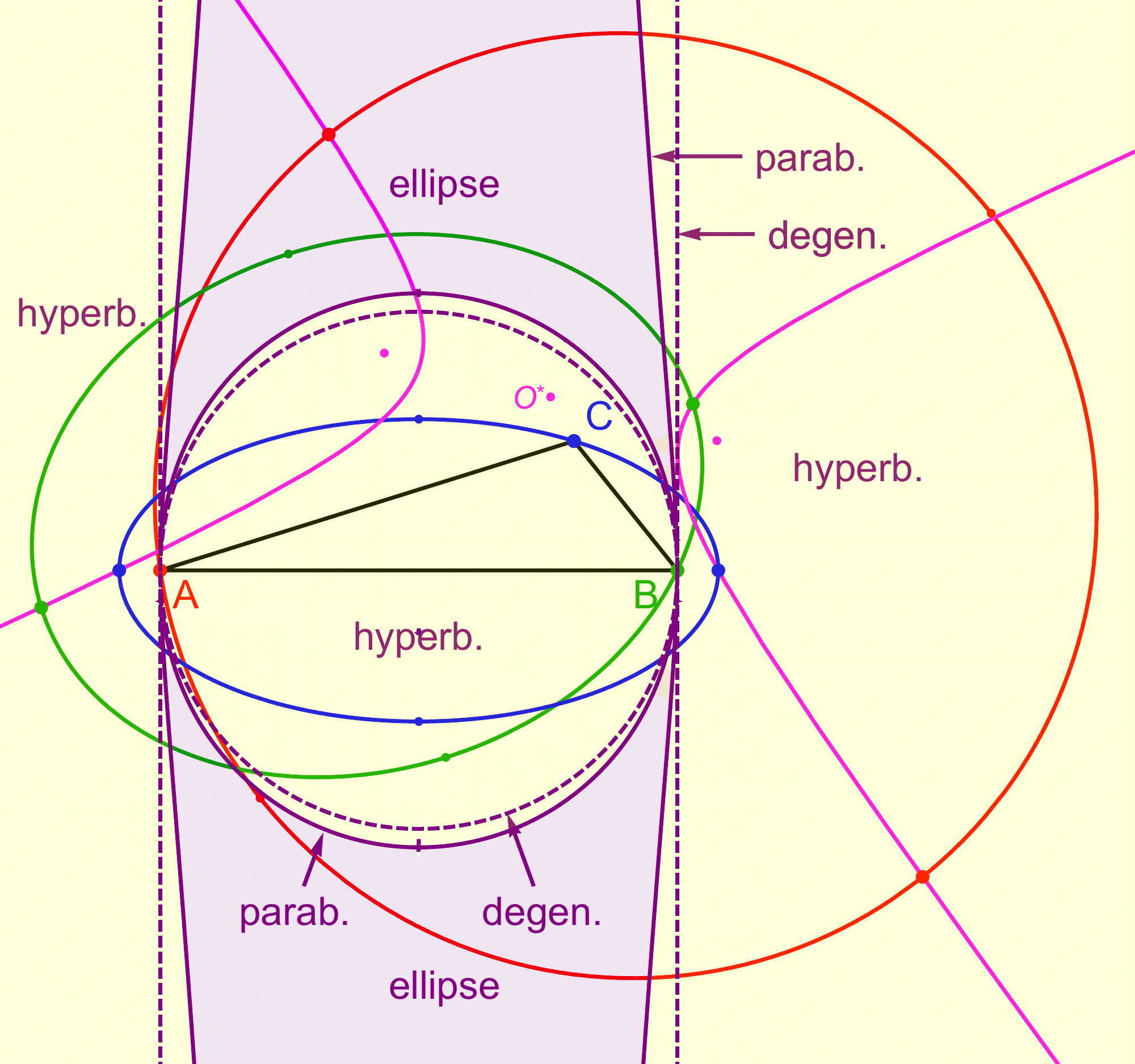}
    \caption{With $A,B$ fixed, the solid (resp. dashed) purple lines are the locus of 
    $C$
    such that the Yiu conic $\Y$ is a parabola (resp. degenerate). As indicated, in between said boundaries, the conic is either an ellipse or a hyperbola. A particular $\triangle ABC$ is shown with $C$ interior to the circumcircle, where $\Y$ is a hyperbola (magenta).}
    \label{fig:vells-zones}
\end{figure}

\subsection*{What about the co-vertices?}
It turns out that for $A,B$ fixed, there is a locus of $C$ such that the 6 co-vertices of the V-ellipses lie on a conic. Without loss of generality, let $A=(-1,0)$, and $B=(1,0)$. Referring to \cref{fig:vells-covertices-non-conic}:

\begin{proposition}
The locus of $C$ such that the 6 co-vertices of $\E_a$, $\E_b$, and $\E_c$ lie on a conic is given by:
{\small
\begin{align*}
 &\left( {x}^{6}-  (  2\,{y}^{2}+3  ) {x}^{4}- \ (  3\,{y}^
{4}-8\,{y}^{2}-3) {x}^{2}+11\,{y}^{4}-6\,{y}^{2}-1 \right) \rho_1\,\rho_2+\\
&\left( -2\,{x}^{6}-  (  22\,{y}^{2}-6  ) {x}^{4}-
  ( 14\,{y}^{4}-36\,{y}^{2}+6  ) {x}^{2}+6\,{y}^{6}+22\,{y}^
{4}-14\,{y}^{2}+2 \right)  \left( \rho_1+\rho_2 \right) +\\
&2\,x \left( {x}
^{6}+  (3\,{y}^{2}-3) {x}^{4}+  ( 3\,{y}^{4}-2\,{y}^{2
}+3) {x}^{2}+{y}^{6}-7\,{y}^{4}-{y}^{2}-1 \right)  \left( \rho_1-\rho_2 \right)+ \\
& 2(  {x}^{2}+ {y}^{2}-1)  \left( 5\,{x}
^{4}+  2({y}^{2}-5) {x}^{2}-3\,{y}^{4}-14\,{y}^{2}+5
 \right)   ( {x}^{2}-1) =0
\end{align*}}
where $\rho_1 = \sqrt{x^2 + y^2 + 2 x + 1}$, and $\rho_2 = \sqrt{x^2 + y^2 - 2 x + 1}$.
\end{proposition}

\begin{proof} Computer algebra system-based manipulation.
\end{proof}

Notice that a full 8 branches of the locus converge on either $A$ or $B$. Also note that if one attempts to eliminate the square roots in the implicit, one obtains a degree-36 polynomial.

Examples of the 6-point co-vertex conic for different locations of $C$ on the above locus appear in \cref{fig:vells-covertices-quad}, suggesting that (i) this conic is always a hyperbola, and that (ii) depending on the branch of
the locus of $C$
is on, co-vertices are split as 3:3 or 5:1 along the two branches of the hyperbola.

\begin{figure}
    \centering
    \includegraphics[width=.8\textwidth,frame]{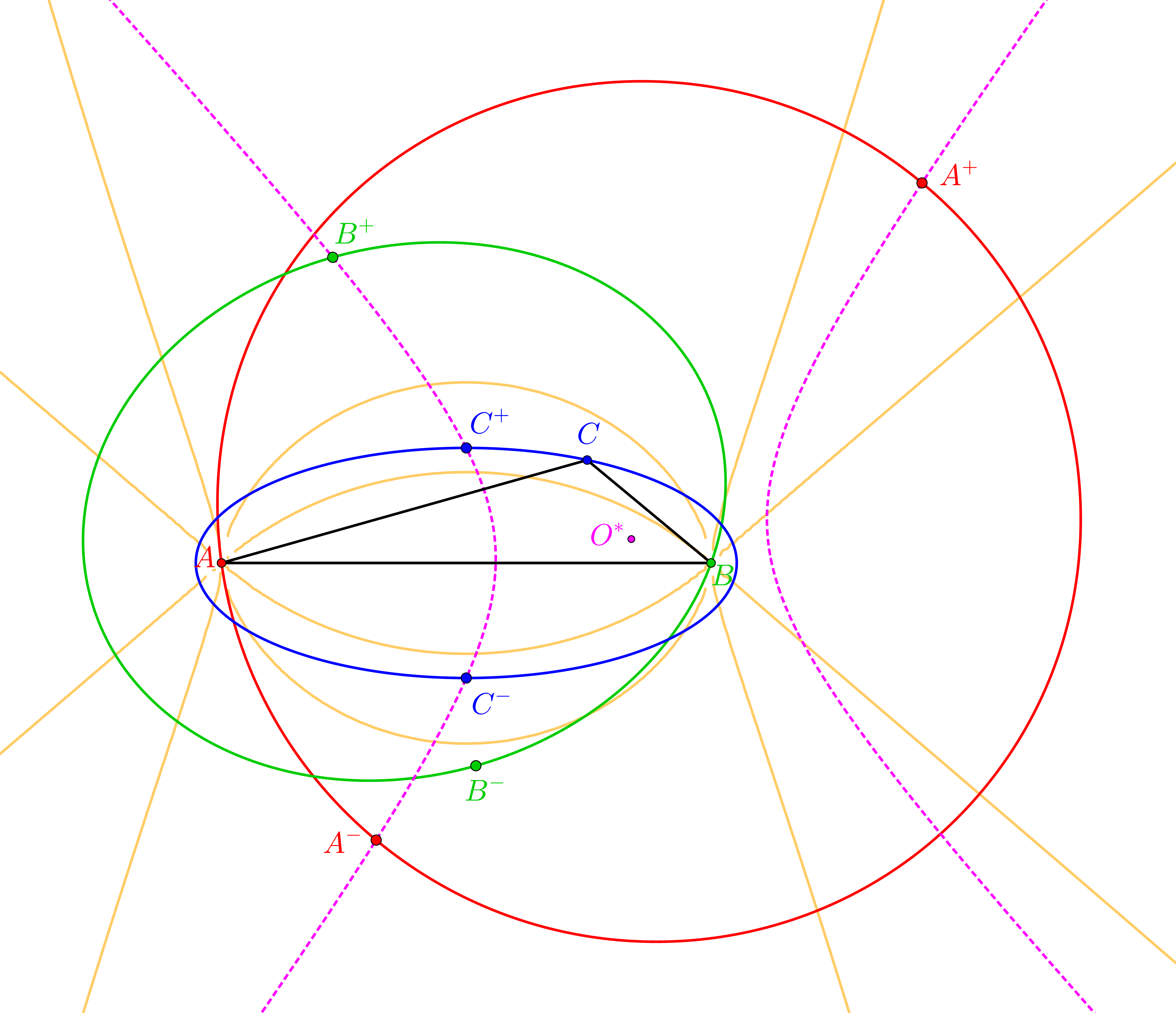}
    \caption{A $\triangle ABC$ is shown, as well as its 3 P-ellipses (red, green, blue) with co-vertices $A^+,A^-$, $B^+,B^-$, $C^+,C^-$. Also shown is the locus of $C$ (yellow) such that the co-vertices lie on a conic. Notice that for the triangle shown, $C$ does lie on said locus. For illustration, a hyperbola is shown (dashed magenta) which passes through 5 co-vertices but misses $B^-$.}
    \label{fig:vells-covertices-non-conic}
\end{figure}

\begin{figure}
\centering
\hrule
\begin{subfigure}[b]{0.48\textwidth}
\centering
 \includegraphics[width=\textwidth]{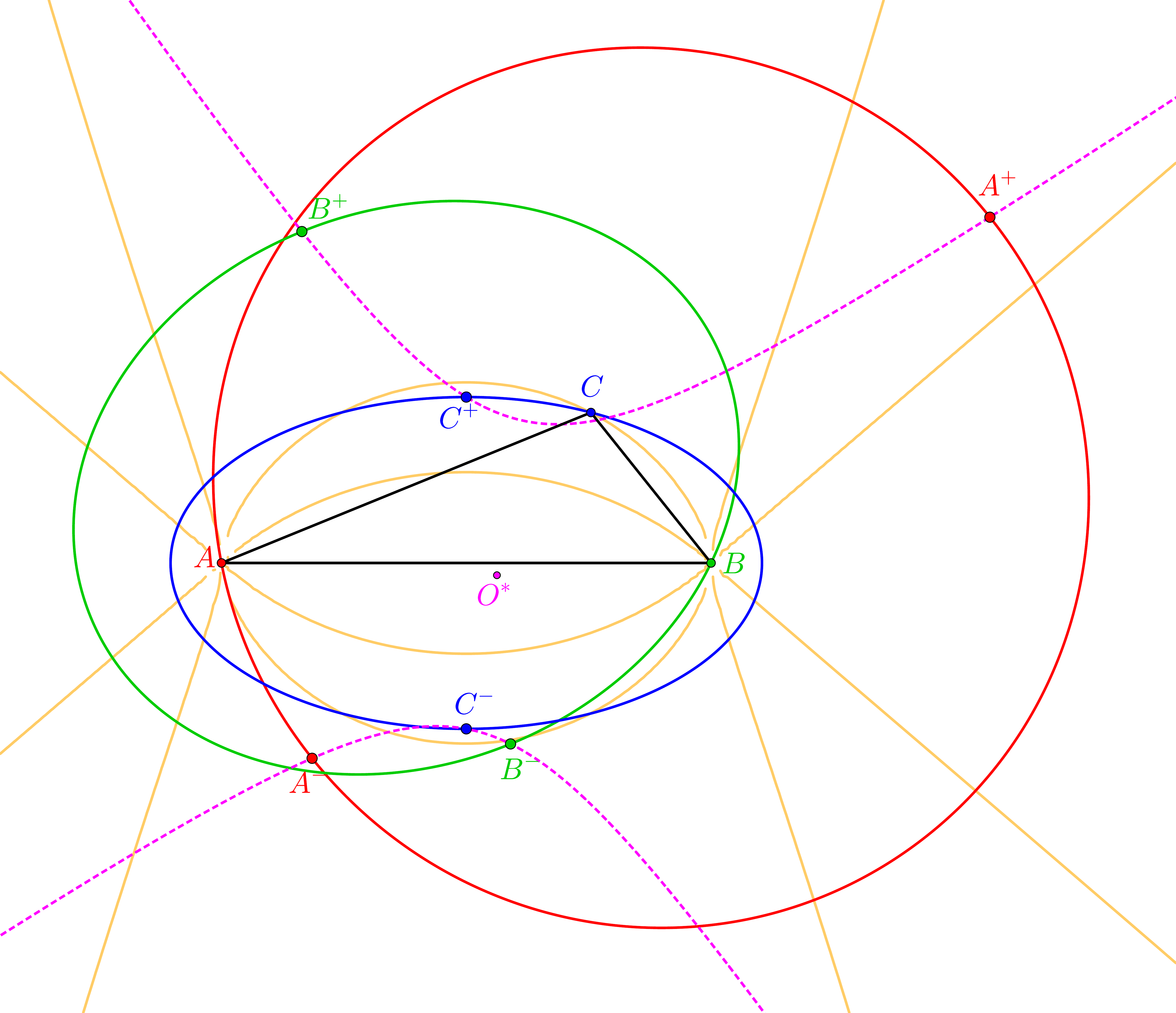}
\end{subfigure}
\rulesep
\begin{subfigure}[b]{0.48\textwidth}  
\centering 
\includegraphics[width=\textwidth]{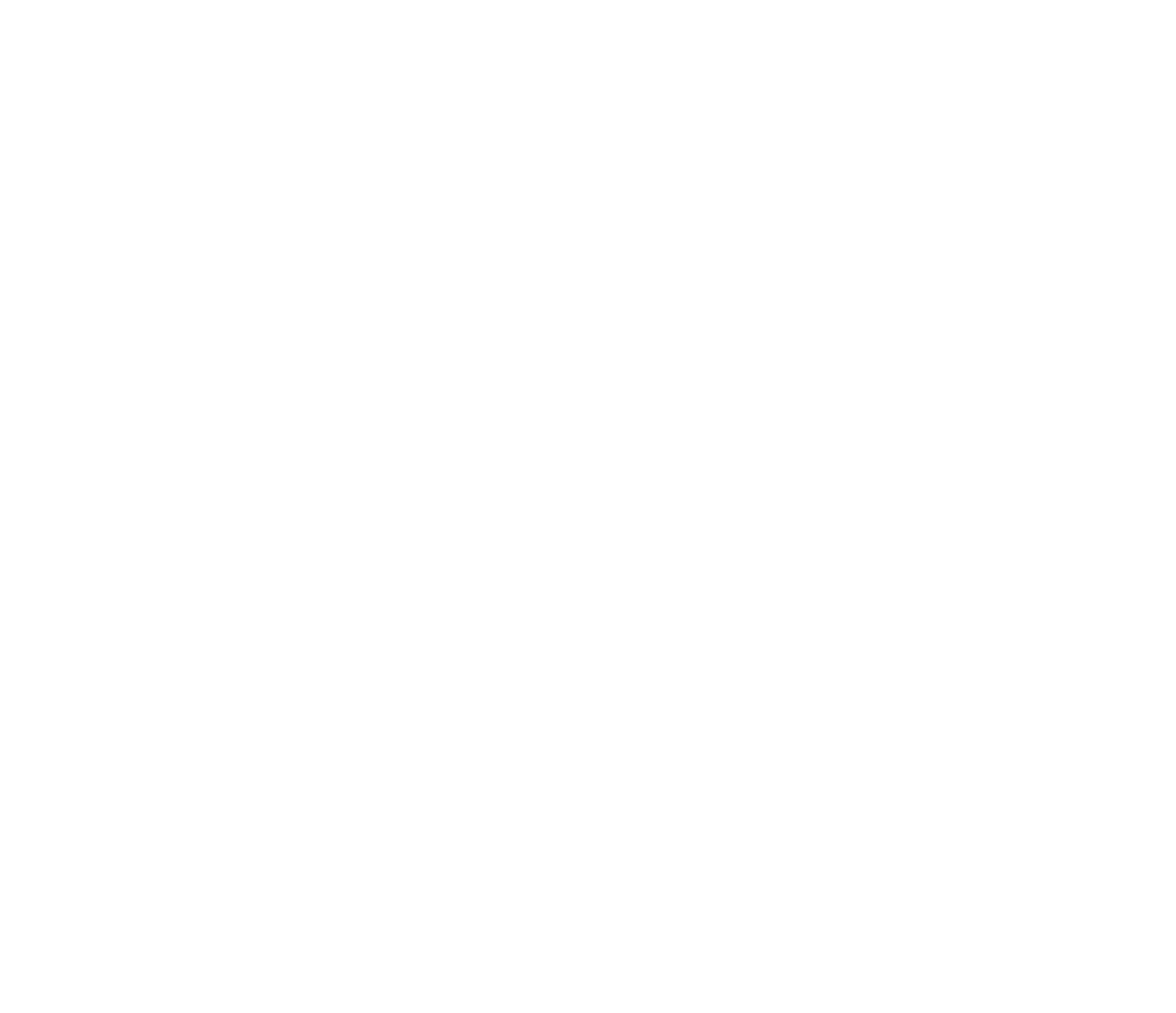}
\end{subfigure}
\hrule
\begin{subfigure}[b]{0.48\textwidth}   
\centering 
\includegraphics[width=\textwidth]{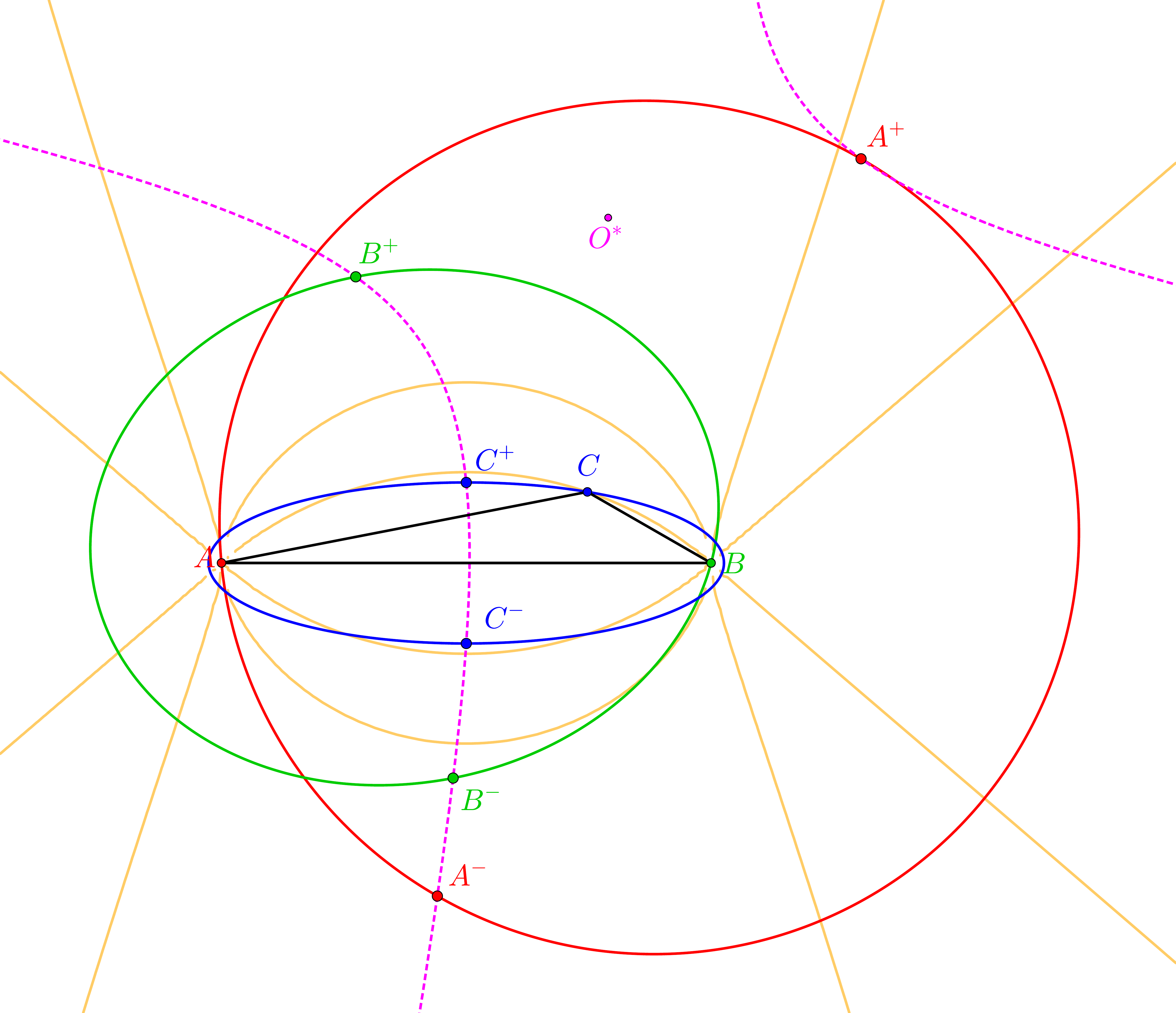}
\end{subfigure}
\rulesep
\begin{subfigure}[b]{0.48\textwidth}
\centering 
\includegraphics[width=\textwidth]{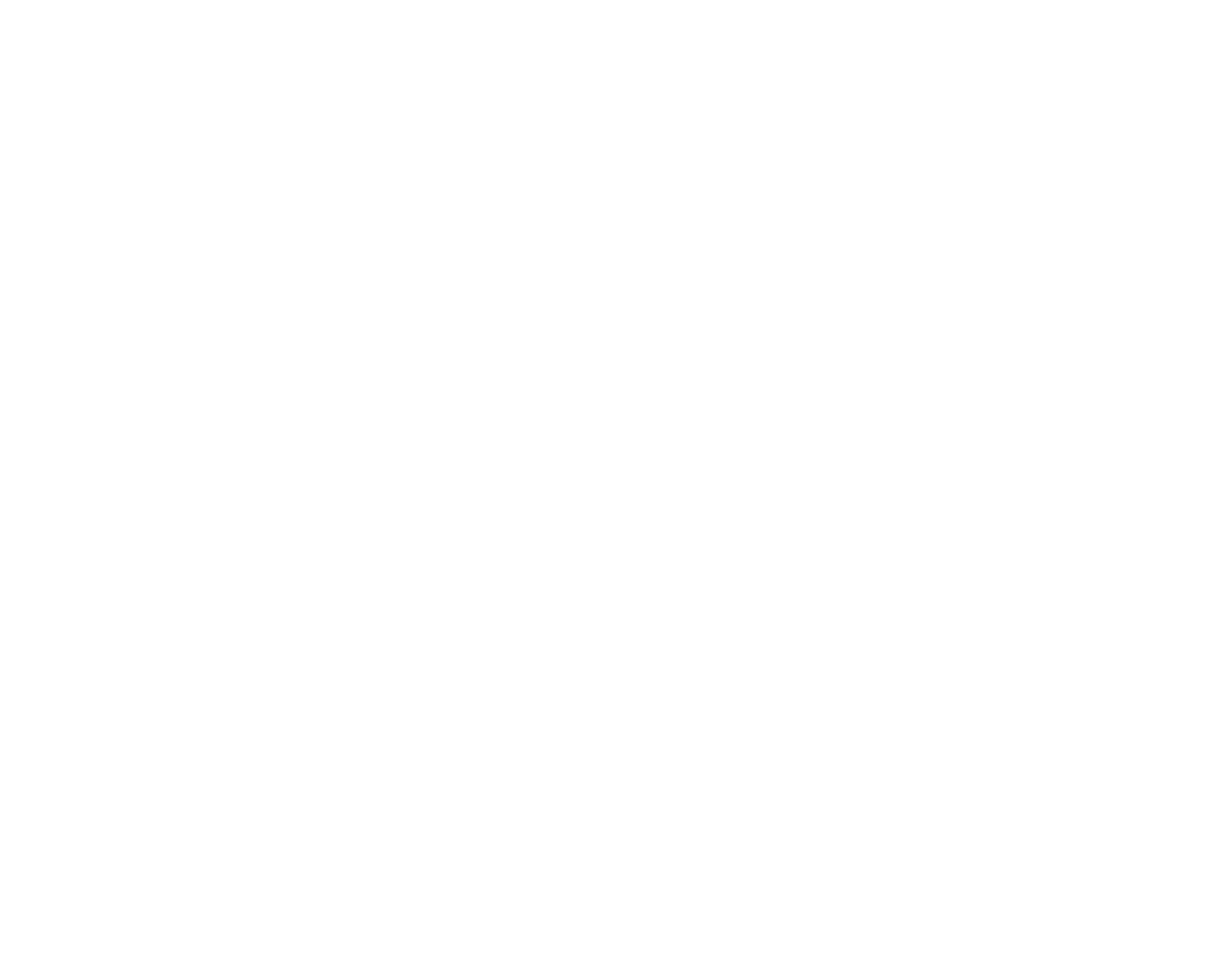}
\end{subfigure}
\hrule
\caption{Four choices for $C$ on the locus (yellow) such that the co-vertices $A^+,A^-$, $B^+,B^-$, and $C^+,C^-$ of V-ellipses (red, green, blue) of $\triangle ABC$ lie on a conic (dashed magenta). In the top two cases (resp. bottom two), the co-vertices are split 3x3 (resp. 5x1) on each branch of the conic.} 
\label{fig:vells-covertices-quad}
\end{figure}

\section{A triad of P-ellipses}
\label{sec:p-ell}
Referring to \cref{fig:ptriad}:

\begin{definition}[P-ellipses]
A triad of P-ellipses $\E_a^*,\E_b^*,\E_c^*$ have foci on $(B,C)$, $(C,A)$, $(A,B)$ and pass through a given point $P$.
\label{def:ptriad}
\end{definition}

Consider a triad of P-ellipses as in \cref{def:ptriad}. 

\begin{figure}
\centering
\includegraphics[width=.8\textwidth,frame]{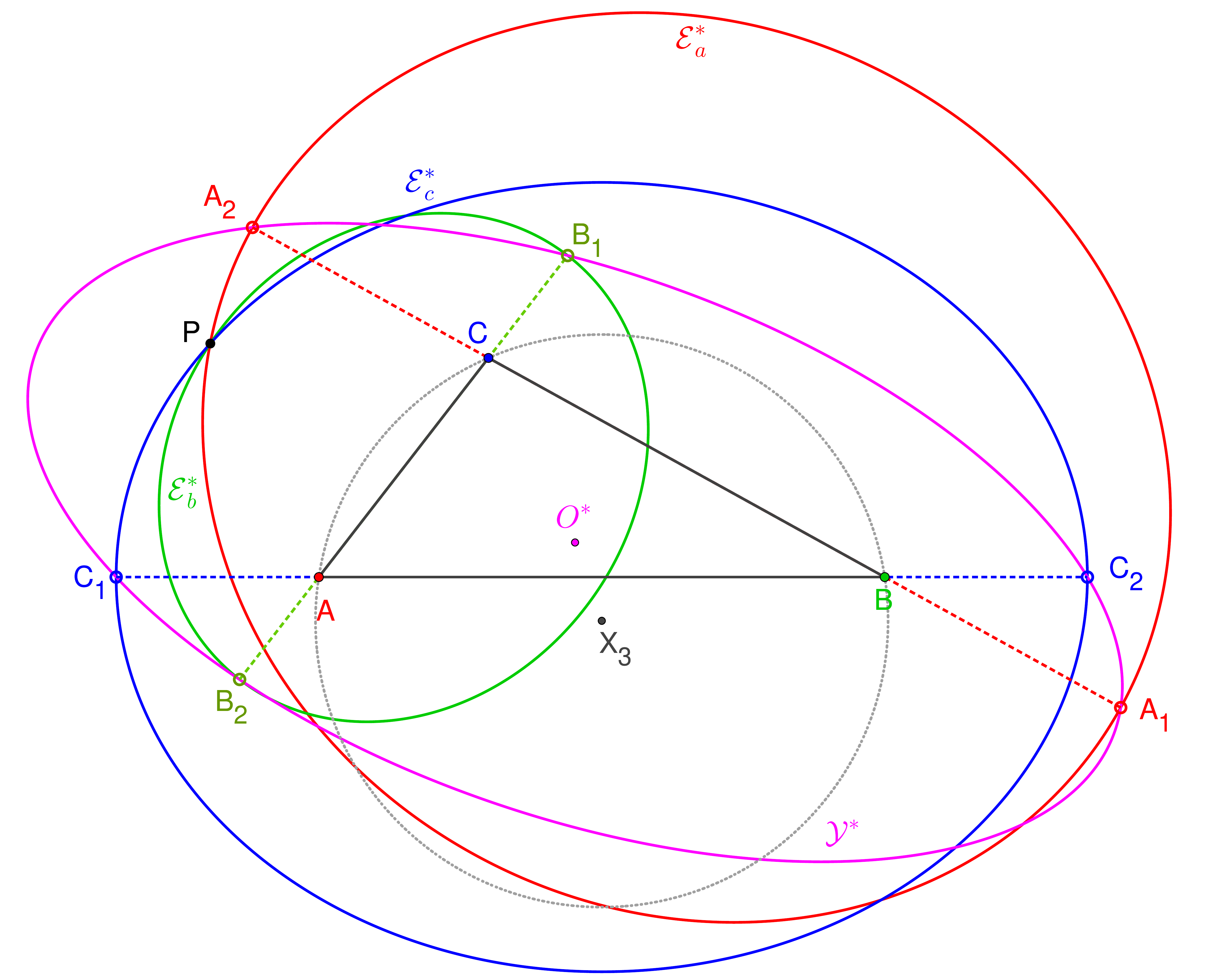}
\caption{The six vertices $A_1$, $A_2$, $B_1$, $B_2$, $C_1$, and $C_2$ of P-ellipses $\E_a^*,\E_b^*,\E_c^*$ are on a conic $\Y^*$. For reference, the circumcenter $X_3$ of $\triangle ABC$ and the center $O^*$ of $\Y^*$ are also shown.}
\label{fig:ptriad}
\end{figure}

\begin{theorem}
The six vertices of a triad of P-ellipses lie on a conic $\Y^*$.
\end{theorem}
\begin{proof} 
Referring to 
\cref{fig:ptriad}, let $A_1,A_2$ (resp. $B_1,B_2$, and $C_1,C_2$) denote the vertices of $\E_a^*$ (resp. $\E_b^*$ and $\E_c^*$). Note that $A_1 A_2$ shares its midpoint with $BC$, and so on cyclically. Therefore:
$AC_2=BC_1$, $BA_2=A_1C$, and $CB_1=B_2A$. To finish the proof, we apply Carnot's theorem as in \cref{prop:yiu3}.
\end{proof}

Let $a,b,c$ denote the sidelengths of $\triangle ABC$. Let $\delta_a=|PB|+|PC|$, $\delta_b=|PC|+|PA|$,
$\delta_c=|PA|+|PB|$. Referring to \cref{fig:circ-anticev}(left).

\begin{figure}
\centering
\begin{subfigure}[c]{0.48\textwidth}
\centering
\includegraphics[trim=30 40 60 40,clip,width=\textwidth]{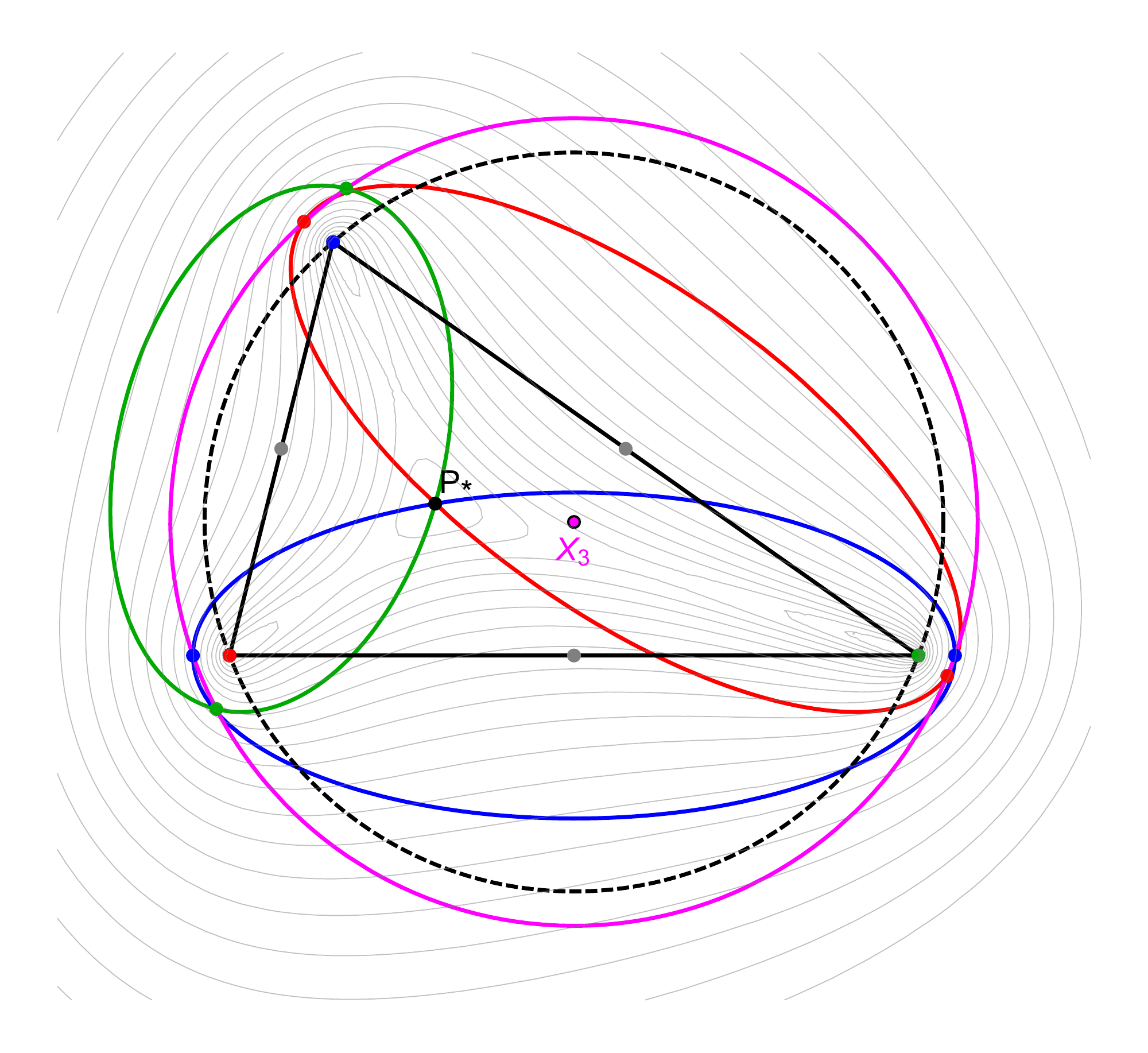}
\end{subfigure}
\rulesep
\begin{subfigure}[c]{0.48\textwidth}
\centering
\includegraphics[trim=0 40 0 0 ,clip,width=\textwidth]{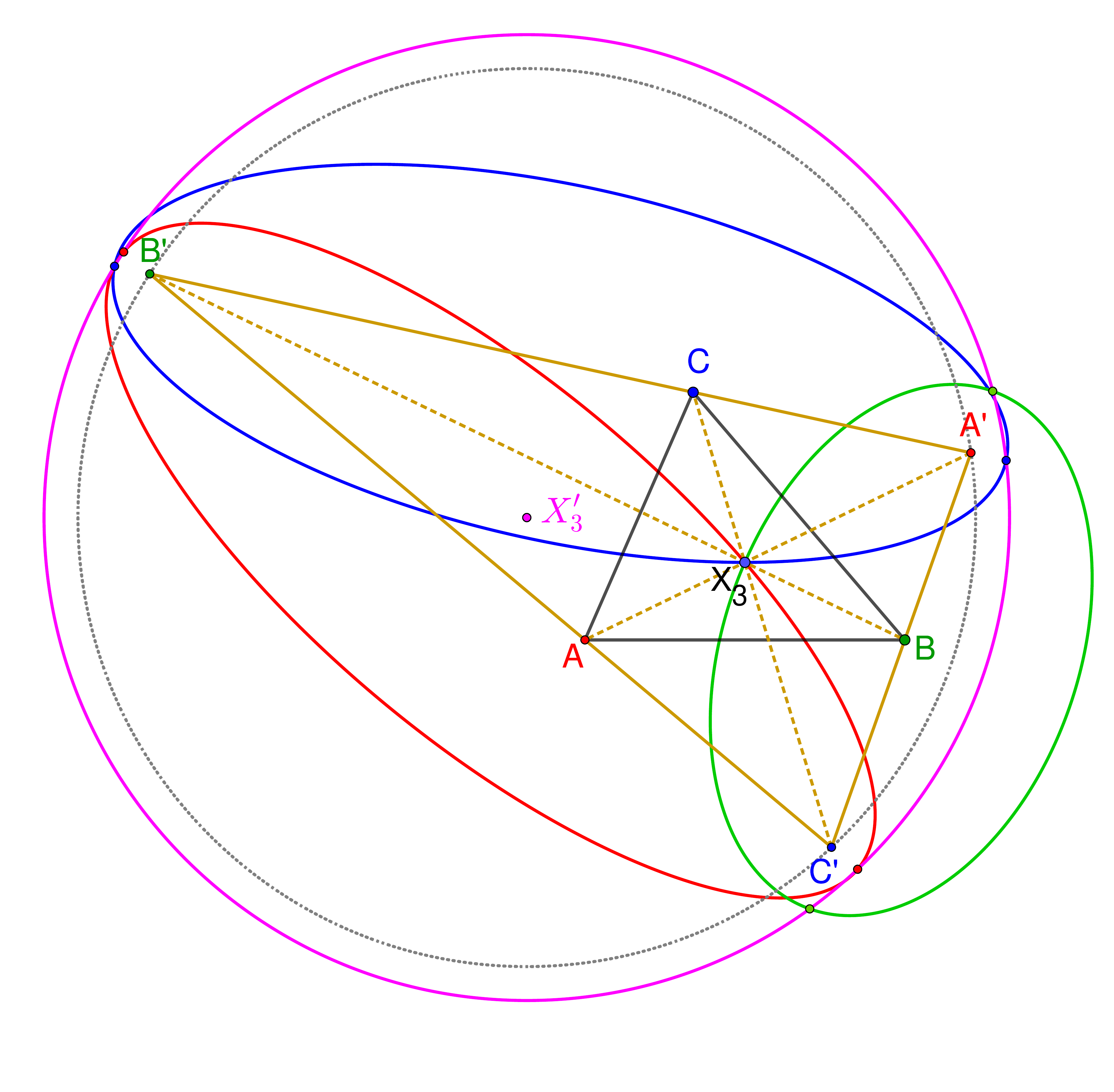}
\end{subfigure}
\caption{\textbf{Left:} Given a triangle, there is a unique $P^*$ such that the 6-point conic (magenta) of a triad of P-ellipses is a circle. The latter is concentric with the circumcircle (dashed black). Also shown are level curves of the functional in \cref{prop:pstar}: $P^*$ is its unique zero. \textbf{Right:} The vertices of a P-ellipse with foci on vertices of the $X_3$-anticevian $\triangle A'B'C'$ of $\triangle ABC$, and passing through the latter's $X_3$ lie on a circle (magenta). The latter is concentric with the circumcircle of $\triangle A'B'C$ (dashed black) at $X_3'$.}
\label{fig:circ-anticev}
\end{figure}


\begin{proposition}
There is a unique point $P^*$ such that $\Y^*$ is a circle given by:
\begin{align*}
&\left[(a^2 - \delta_a^2)(c^2 -b^2 + \delta_b^2 - \delta_c^2)\right]^2 +\\
&\left[(b^2 - \delta_b^2)(a^2 - c^2 + \delta_c^2 - \delta_a^2)\right]^2+\\
&\left[(c^2 - \delta_c^2)(b^2 -a^2 + \delta_a^2 - \delta_b^2)\right]^2  = 0
\end{align*}
Furthermore, $\Y^*$ is concentric with the circumcircle of $\triangle ABC$.
\label{prop:pstar}
\end{proposition}
\noindent Level curves of the above function for a particular triangle are shown in \cref{fig:circ-anticev}(left). Interestingly, there is a straightforward way to construct a triangle whose $\Y^*$ is a circle.

\begin{definition}[anticevian triangle]
Given $\triangle ABC$ and a point $Q$, the {\em $Q$-anticevian} $\triangle A'B'C'$ is such that $\triangle ABC$ is its  $Q$-cevian \cite{mw}.
\end{definition}

\noindent Referring to \cref{fig:circ-anticev}(right), a first ``needle in a haystack'' find is:

\begin{proposition}
Given a reference triangle $\triangle ABC$, its $X_3$ is the $P^*$  
of its $X_3$-anticevian $\triangle A'B'C'$. Furthermore, (i) $\Y^*$ of the the latter is concentric with its circumcircle, and (ii) its center lies on the $X_4 X_6$ line of $\triangle ABC$. 
\label{prop:anticevian}
\end{proposition}

\begin{proof}
This needle-in-a-haystack phenomenon was discovered experimentally and then verified using CAS.
\end{proof}
\noindent Barycentric coordinates for the circumcenter $X_3'$ of the $X_3$-anticevian appear in \cref{app:barys}.

While it can be shown that given a generic $\triangle A'B'C'$, there is always a triangle $\triangle ABC$ which is the former's $X_3$-cevian (map is invertible), we don't yet have a geometric construction for the latter. 

\subsection*{A degenerate 6-point conic:} As shown in  \cref{fig:pell-locus-equi}(left), a simple condition renders $\Y^*$ degenerate, namely:

\begin{proposition}
If $P$ is on the circumcircle of $\triangle{ABC}$, 
$\Y^*$ is degenerate (two straight lines).
\end{proposition}

\begin{proof}
Via CAS, it can be verified that the 3x3 discriminant of the homogeneous equation for the conic vanishes.
\end{proof}


\noindent Referring to \cref{fig:pell-locus-equi}(left):

\begin{figure}
\centering
\begin{subfigure}[c]{0.48\textwidth}
\centering
\includegraphics[width=\textwidth]{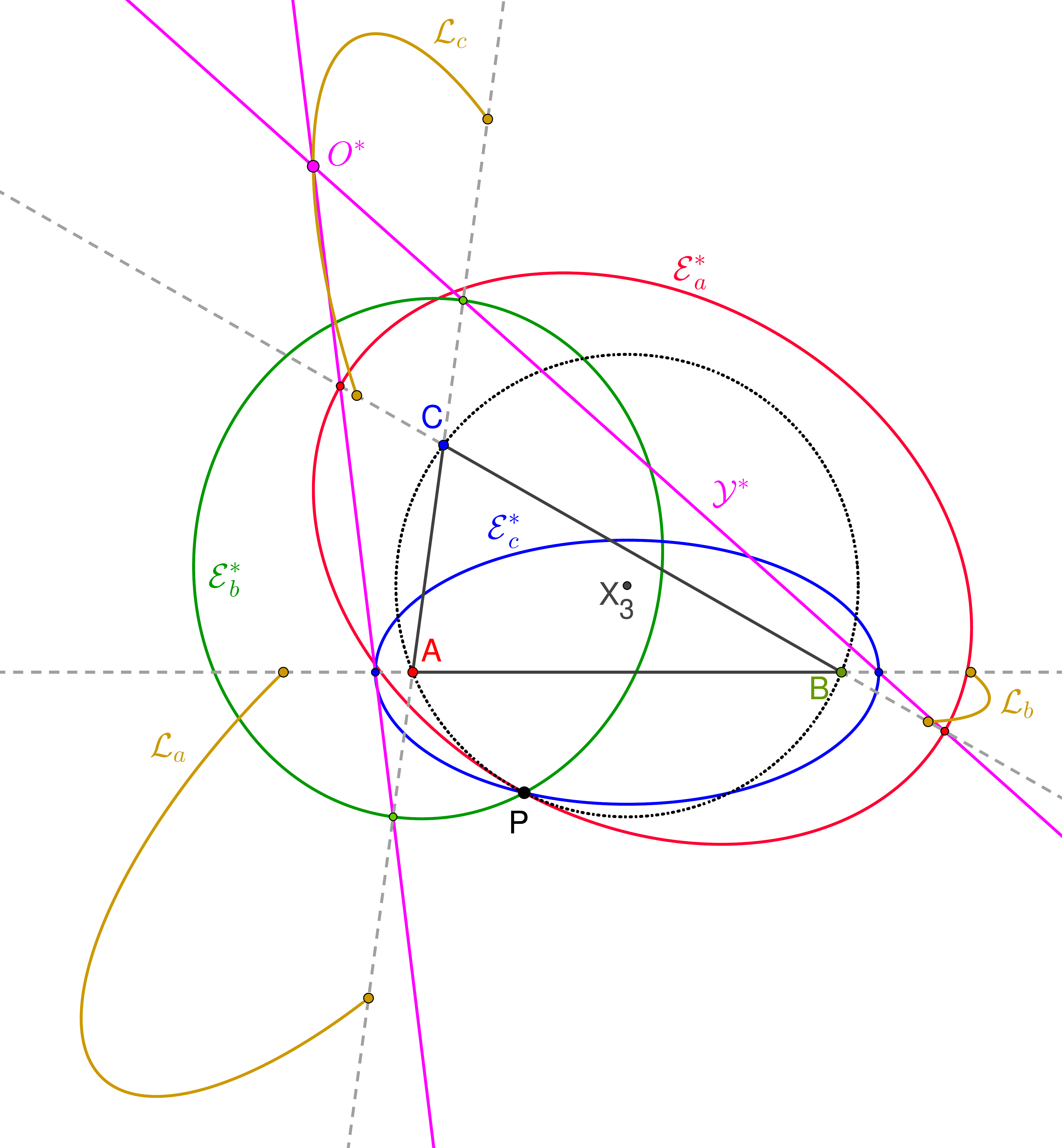}
\end{subfigure}
\rulesep
\begin{subfigure}[c]{0.48\textwidth}
\centering
\includegraphics[trim=0 0 0 0,clip,width=\textwidth]{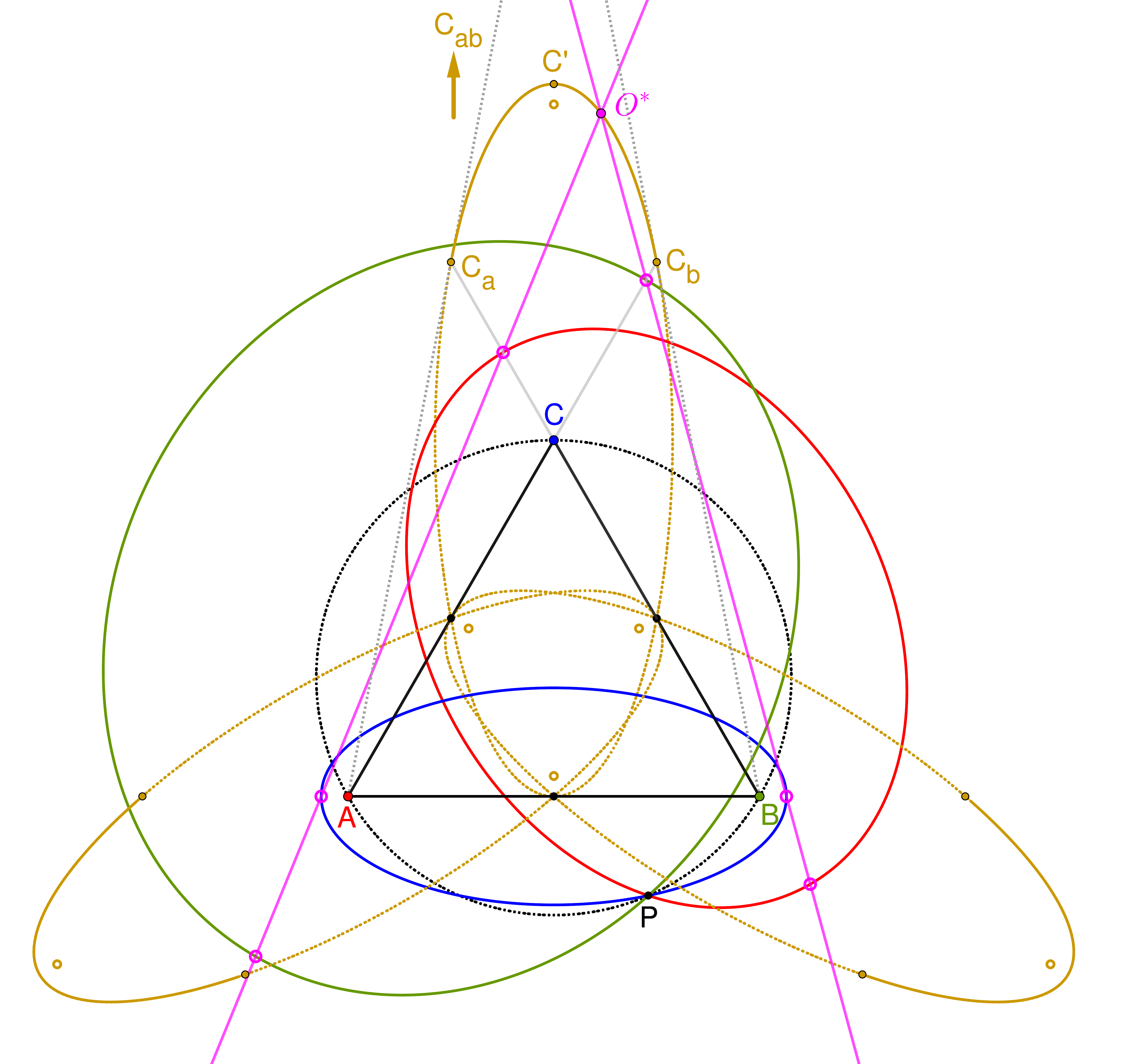}
\end{subfigure}
\caption{\textbf{Left:} If $P$ lies on the circumcircle of $\triangle ABC$, the Yiu conic $\Y^*$ (magenta) is degenerate. Over $P$ on the circumcircle, the locus (gold) of the center $O^*$ of the degenerate conic (magenta lines) is the union of three arcs of ellipse $\L_a,\L_b,\L_c$.\textbf{Right:} Over all $P$ on the circumcircle of an equilateral $\triangle ABC$, the locus of the center $O^*$ of the degenerate 6-pt conic (magenta) is the union of 3 elliptic arcs (solid gold) centered on $A,B,C$, whose major axes are the altitudes of $\triangle ABC$. The major (resp. minor) semi-axes measure $|AB|=\sqrt{3}/2$ (resp. $|AB|=\sqrt{3}/6$).}
\label{fig:pell-locus-equi}
\end{figure}



\begin{proposition}
Over $P$ on the circumcircle, the locus of the center $O^*$ of $\Y^*$ is the union of arcs of three distinct ellipses $\L_a,\L_b,\L_c$, all of which pass through the midpoints of $ABC$. The endpoints of $\L_a$ are one vertex of V-ellipse $\E_b$ and one of $\E_c$, and so on cyclically for the endpoints of $\L_b,\, \L_c$. 
\label{prop:3-arcs}
\end{proposition}

\begin{proof}
Referring to 
\cref{fig:pell-locus}, that the endpoints of $\L_c$ are a vertex $A''$ of $\E_a$ and a vertex $B''$ of $\E_b$ can be seen from the fact that the limit of $\E_a^*$ (resp. $\E_b^*$) as $P$ approaches $A$ (resp. $B$) is $\E_a$ (resp. $\E_b$) and that the center $O^*$ of the degenerate $\Y^*$ will approach the intersection of $A A''$ and $B C$. The same argument applies for the endpoints of $\L_a,\L_b$, cyclically. To show that the locus of $O^*$ is the union of three elliptic arcs, we (i) restrict $P$ to a given ``third'' of the circumcircle, e.g., the arc between $A$ and $B$. Then (ii) we obtain, via a CAS, a (rather long) symbolic expression for the implicit function $f(x,y)$ representing the ellipse which passes through the 5 proposed points, namely, two vertices of V-ellipses and the midpoints of the sides of $\triangle ABC$. We then (iii) obtain a parametric expression for $O^*$ as a function of $P$ and plug it into $f(x,y)$, and notice via a CAS, that this simplifies to zero, independent of $P$. (iv) The same can be repeated cyclically for the other 3 portions of the circumcircle.
\end{proof}


\begin{figure}
\centering
\includegraphics[width=.8\textwidth,frame]{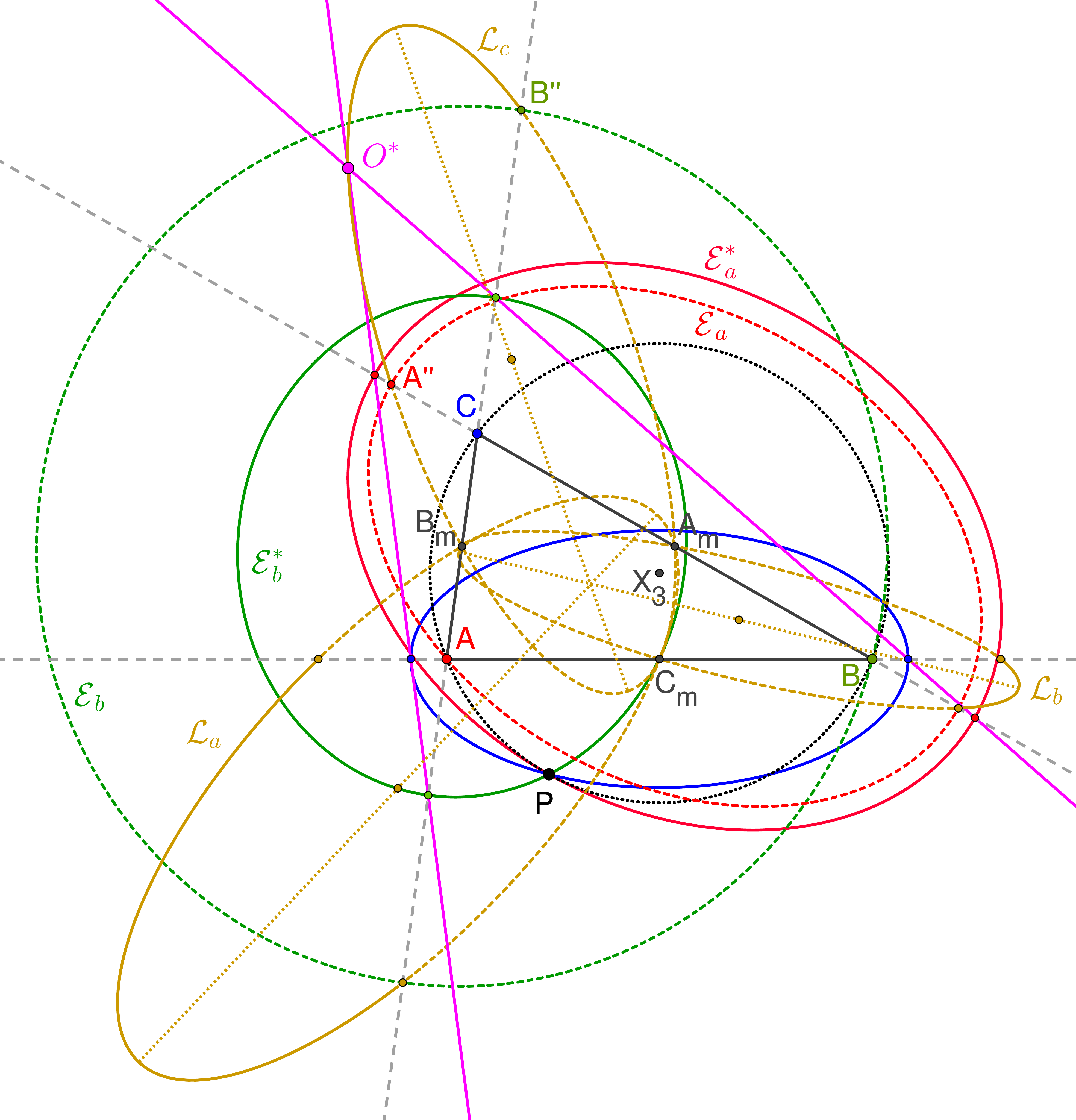}
\caption{Definitions used in \cref{prop:3-arcs}. The locus of  $O^*$ is the union of three arcs of ellipse (solid gold) $\L_a,\L_b,\L_c$, each of which passes through the 3 midpoints $A_m,B_m,C_m$ of $\triangle ABC$. The endpoints $A'',B''$ of $\L_c$ are vertices of V-ellipses $\E_a$ and $\E_b$ (dashed red, green). The major axes (dashed gold) of the three loci nearly concur, though not exactly.}
\label{fig:pell-locus}
\end{figure}


\noindent Referring to \cref{fig:pell-locus-equi}(right): 
 
\begin{corollary} Let $\triangle ABC$ be an equilateral of side 1.
Over $P$ on the circumcircle, the locus of the center $O^*$ of the degenerate $\Y^*$ is the union of arcs of three congruent ellipses with semi-axes $a=\sqrt{3}/2$ and $b=\sqrt{3}/6$, centered on $A,B,C$.
\end{corollary}

 Let $C_a$ and $C_b$ denote the endpoints of the elliptic locus of $O^*$, over $P$ on the arc of the circumcircle below $AB$. Let $C'$ denote the locus' top vertex. Referring to \cref{fig:pell-locus-equi}(right), the following can be shown:

\begin{compactitem}
\item $C_a$ and $C_b$ are the reflections of the midpoints of $AC$ and $BC$ about $C$
\item lines $AC_a$ and $BC_b$ are tangent to the locus. Let $C_{ab}$ denote their intersection.
\item  $C'$ is the midpoint between $C$ and $C_{ab}$.
\item Therefore, $C_a C_b$ is the mid-base of $\triangle  A C_{ab} B$, therefore the latter is 3 times the area of $\triangle ABC$.
\end{compactitem}

\subsection*{Regions of conic type:}
It turns out the type of $\Y^*$ (ellipse, parabola, hyperbola, degenerate) depends on the position of $P$. The case of an equilateral $\triangle ABC$ is illustrated in \cref{fig:equi}.

\begin{remark}
If $\triangle ABC$ is an equilateral, it can be shown that the portions of the locus of $P$ such that $\Y^*$ is: (i)
degenerate (deltoid interior to $\triangle ABC$) are branches of 3 regular cubics; (ii) a parabola: branches of a 
degree-20 polynomial on $x,y$. 
\end{remark}

\begin{figure}
\centering
\includegraphics[trim=0 0 0 0,clip,width=.6\textwidth,frame]{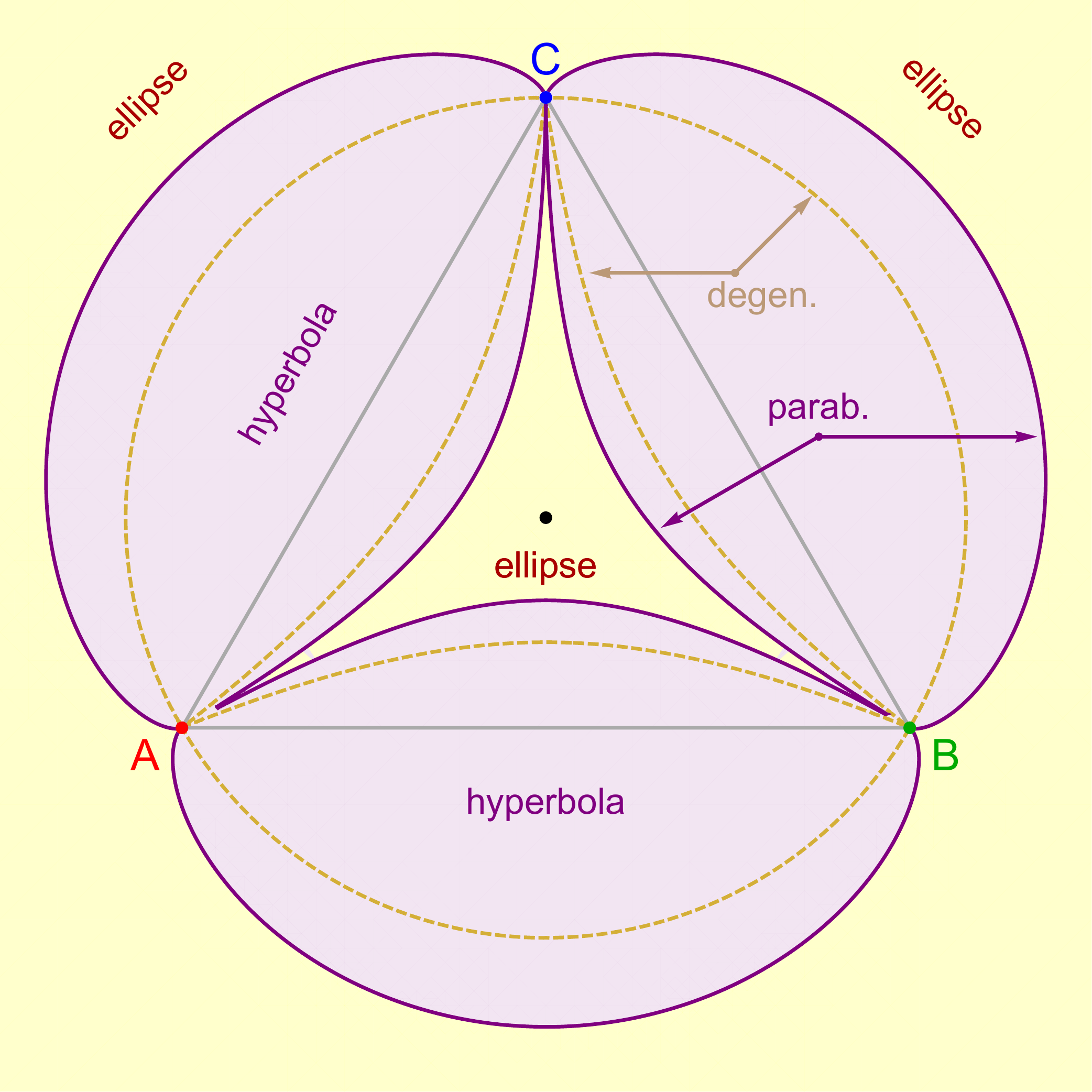}
\caption{For $\triangle ABC$ an equilateral, the figure illustrated regions of $P$ such that the $\Y^*$ conic is of a given type.}
\label{fig:equi}
\end{figure}

\begin{remark}
If $\Y^*$ is a hyperbola it can never be a rectangular one.
\end{remark}

\subsection*{What about the co-vertices?}

It turns out that for given $\triangle ABC$, there is a 1d locus for $P$ such that the 6 co-vertices lie on a conic. As before, let  Let $\delta_a=|PB|+|PC|$, $\delta_b=|PC|+|PA|$,
$\delta_c=|PA|+|PB|$.
Referring to \cref{fig:pells-covertices}:
\begin{proposition}
If $\triangle ABC$ is an equilateral, the locus for $P$ such that the 6 co-vertices lie on a conic $\Y^{\dagger}$ is given by:
\begin{align*} 
\delta_a^{2}\delta_b^{2}
+\delta_a^{2}\delta_c^{2}+\delta_b^{
2}\delta_c^{2}-8(\,\delta_a^{2}+\,\delta_b^{2}+\,\delta_c^{2})
+48=0
\end{align*}
Furthermore, the center $O^{\dagger}$ of $\Y^{\dagger}$ lies on the incircle of the equilateral.
\end{proposition}

Note: if one eliminates all square roots involved in computing $\delta_a,\delta_b,\delta_c$, the above becomes a degree-10 equation on $x,y$.

\begin{figure}
\centering
\includegraphics[width=.7\textwidth,frame]{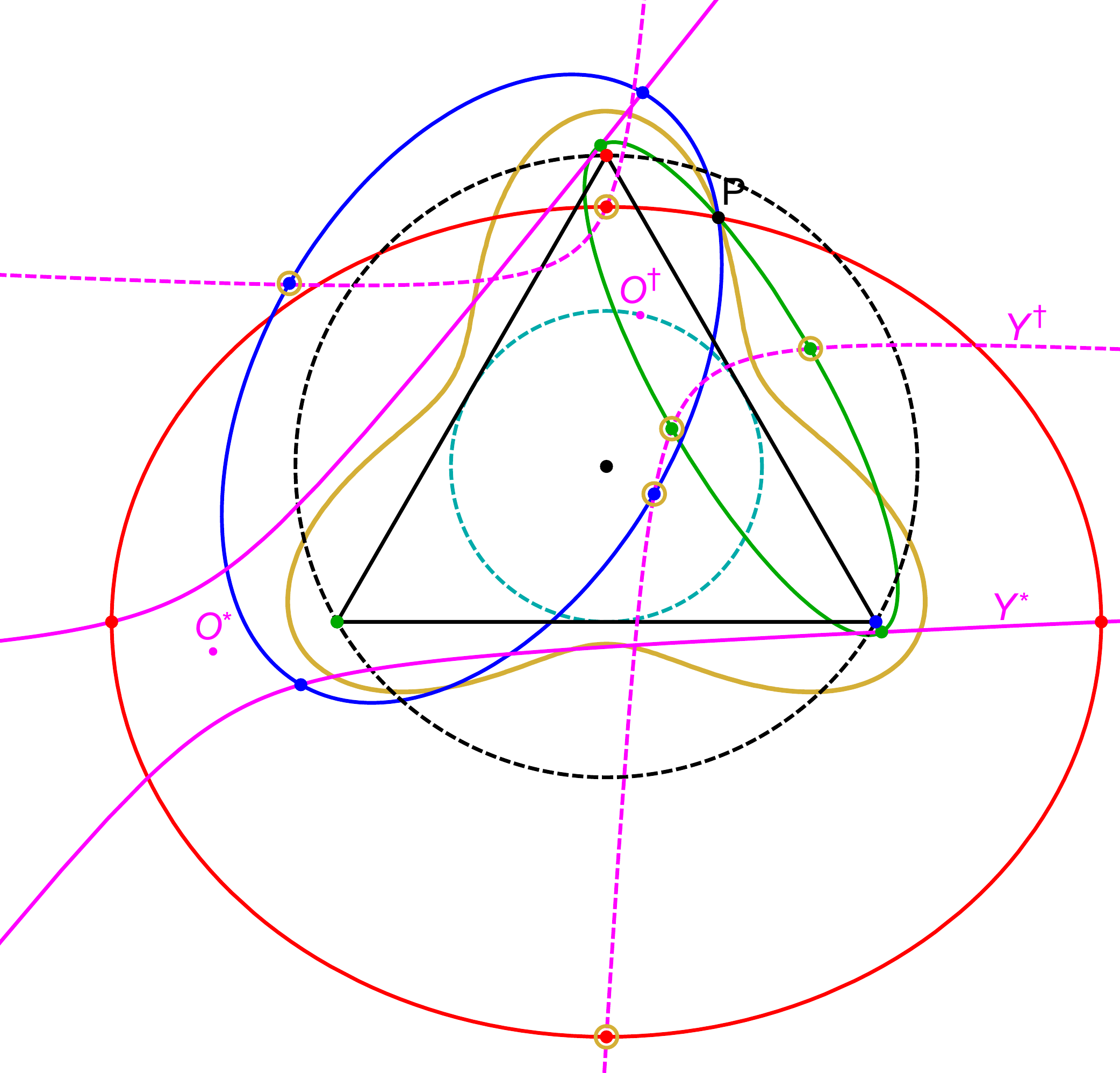}
\caption{Given an equilateral (black), the locus for $P$ such that the 6 co-vertices of the 3 P-ellipses lie on a conic is a degree-10 algebraic curve (gold) woven symmetrically about the equilateral (there is an isolated point at the centroid as well). The three P-ellipses (red, green, blue) are shown for a specific choice of $P$ on said locus. Also shown are (i) the conic $\Y^*$ (solid magenta, center $O^*$) through the major vertices, and (ii) the conic $\Y^{\dagger}$ (dashed magenta, center $O^{\dagger}$) through the 6 co-vertices (highlighted by small gold circles). Notice that if $P$ is on the locus, $O^{\dagger}$ lies on the incircle of the equilateral.}
\label{fig:pells-covertices}
\end{figure}


\section{A triad of V-hyperbolas}
\label{sec:v-hyp}
In this section we describe properties -- some old, some new -- of a special triad of hyperbolas, described in \cite[Sec. 12.4, p. 148]{yiu2001-intro} where they are called ``Soddy'' hyperbolas. Referring to \cref{fig:v-hyps}:

\begin{definition}[V-hyperbolas]
Given a triangle $\triangle ABC$, a triad of V-hyperbolas $\H_a,\H_b,\H_c$ have foci on $(B,C)$, $(C,A)$, $(A,B)$ and pass through $A$, $B$, and $C$, respectively.
\end{definition}

\begin{figure}
\centering
\includegraphics[trim=0 0 0 0,clip,width=.9\textwidth,frame]{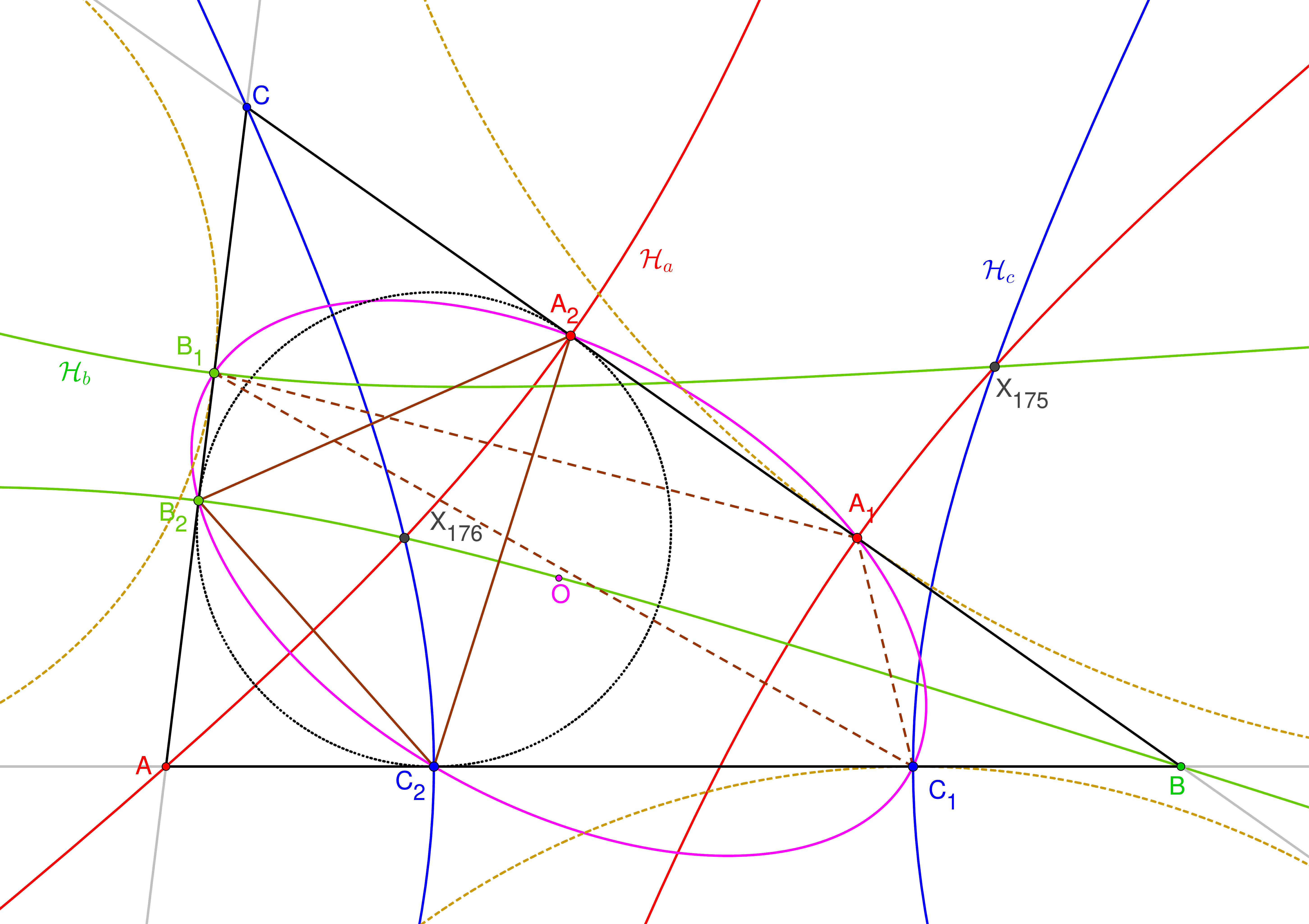}
\caption{A triad of V-hyperbolas $\H_a,\H_b,\H_c$ (red, green, blue) is shown with foci on $(B,C)$, $(C,A)$, and $(A,B)$ passing through $A,B,C$, respectively. Notice (i) their vertices taken as triples $A_1 B_1 C_1$ and $A_2 B_2 C_2$ are the vertices of the extouch (dashed brown) and intouch (solid brown) triangles; (ii) these 6 points are known to lie on the Privalov conic (magenta), whose center $O$ is $X_{5452}$ on \cite{etc}; (iii) the 3 hyperbolas pass through both the ``isoperimeteric'' and ``equal detour'' points, i.e., $X_{175}$ and $X_{176}$, respectively. Note: these coincide when the outer Soddy circle is external to the three mutually tangent circles.}
\label{fig:v-hyps}
\end{figure}

Let $A_1,A_2$ be the vertices of $\H_a$. Define $B_1,B_2$ and $C_1,C_2$ for $\H_b,\H_c$, respectively. Recall the extouch (resp. intouch) triangle is where the 3 excircles (resp. incircle) touch a triangle's sides.

\begin{remark}
Let $\lambda_a = |AB|-|AC|$. In barycentric coordinates for the vertices of $\H_a$ are given by:
$A_1=[0, a + \lambda_a, a - \lambda_a]$, and $A_2=[0, a - \lambda_a, a - \lambda_a]$, with the others computed cyclically.
\end{remark}

\begin{corollary}
$\triangle A_1 B_1 C_1$ (resp. $\triangle A_2 B_2 C_2$) is the extouch (resp. intouch) triangle of $\triangle ABC$.
\label{prop:extouch-intouch}
\end{corollary}
\noindent Recall that for any triangle, the intouch and extouch triangles have the same area \cite[extouch triangle]{mw}. Referring to \cite[X(5452)]{etc}:

\begin{corollary}
$A_1,A_2,B_1,B_2,C_1,C_2$ lie on the Privalov conic centered on $X_{5452}$, and whose barycentric coordinates $x,y,z$ satisfy:
\[
 k_1 k_2 k_3 (x^2 + y^2 + z^2) +  2 \left[ 
 k_2 (k_4 - 2 a b) x y +
 k_3 (k_4 - 2 a c) x z -
 k_1 (k_4 - 2 b c) y z \right] = 0
\]
where $a,b,c$ are the sidelengths of $\triangle ABC$, $k_1 = (a - b - c)$, $k_2 = (a + b - c)$, $k_3 = (a - b + c)$, and $k_4 = a^2 + b^2 + c^2$.
\end{corollary}


\begin{remark}
When $\triangle{ABC}$ is isosceles, one of the V-hyperbolas is degenerate, namely, a pair of coinciding lines at the perpendicular bisector of the base. In this case, the Privalov conic is tangent to the base at its midpoint.
\end{remark}

\subsection*{Intersections between V-hyperbolas:}

Referring to \cref{fig:soddy1}, recall that given a triangle, one can construct\footnote{These pass through the vertices of the intouch triangle.} three ``kissing'' circles $\mathcal{C}_A$,
$\mathcal{C}_B$, and  $\mathcal{C}_C$
centered each on each vertex, and externally tangent to each other \cite{soddy1936}. 

The Apollonius' problem for this triple has (as usual) eight distinct solutions, two of which have the same tangency type (tangent externally or internally to all three circles). 

\begin{definition}[Soddy circles of a triangle]
The two solutions for the Apollonius' problem with the same tangency type are the so-called ``Soddy circles''. The inner Soddy circle is the one whose center is inside the triangle and whose interior does not intersect any of the three kissing circles; the other one is the outer Soddy circle. 
\label{def:soddy}
\end{definition}

Note that the outer Soddy circle, always tangent to the 3 kissing circles, can either (i) contain them (see \cref{fig:soddy1}), (ii) be a line tangent to them, or (iii) be externally tangent to them. For (ii) and (iii) see  \cref{fig:soddy3}.

The centers of Soddy circles correspond to a pair of triangle centers found in \cite{etc} and derived in \cite{yff2007-isoperim}. Namely:

\begin{definition}[Isoperimetric point] The center of the outer Soddy circle ($X_{175}$ in \cite{etc}). Equivalently, the unique point $X$ such that:
\[ |XB| + |XC| \pm |BC| = |XC| + |XA| \pm |CA| = |XA| + |XB| \pm |AB|\]
where the positive (resp. negative) sign is chosen if the outer Soddy circle contains (resp. is external to)   the three mutually tangent circles in \cref{def:soddy}. As derived in \cite{yff2007-isoperim}, containment corresponds to:
\[
\tan{\frac{A}{2}}+\tan{\frac{B}{2}}+\tan{\frac{C}{2}}<2 \]
\end{definition}

In \cite{yff2007-isoperim} it is shown that if the sum of half-tangents is exactly 2, then the outer Soddy circle degenerates to a line. Referring to  \cref{fig:half-tangs}, it can be shown that:

\begin{figure}
\centering
\includegraphics[width=.6\textwidth]{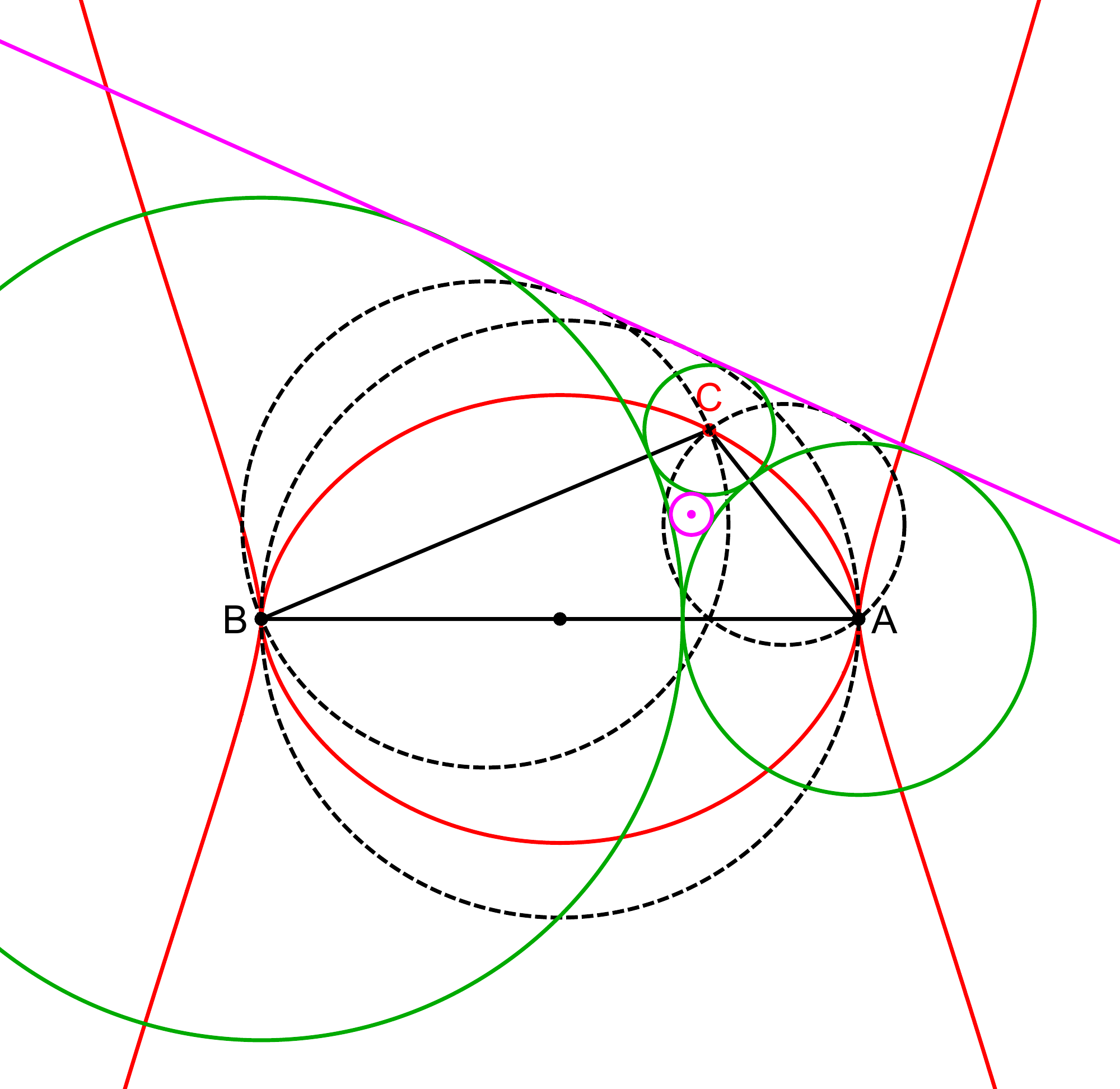}
\caption{When  $A$ and $B$ are fixed, the locus of $C$ (red)  such that the outer Soddy circle degenerates to a line (magenta) is given by the degree-6 implicit equation in \cref{prop:half-tangs}. This line is also tangent to the 3 circles (dashed black) whose diameters are the sides of $\triangle ABC$.}
\label{fig:half-tangs}
\end{figure}

\begin{proposition}
Without loss of generality, let $A=(-1,0)$, $B=(1,0)$, the locus of $C$ such that the sum of half-tangents of $\triangle ABC$ is 2 is given by the union of the following degree-6 polynomial and its reflection about the $x$-axis:
{\small \[
-4 x^6-4 x^4 (2 y^2 + 2 y+1)-4 x^2 (y^4 + y^3 -4 y-5)  + 4 y^5+13 y^4 + 20 y^3+8 y^2 - 8 y- 12=0\]}
\label{prop:half-tangs}
\end{proposition}

\begin{definition}[Equal detour point] The center of the inner Soddy circle ($X_{176}$ in \cite{etc}), always internal to a triangle. Also the unique point $X$ in $\triangle ABC$ such that:
\[|XB| + |XC| -|BC| = |XC| + |XA| - |CA| = |XA| + |XB| - |AB|\]
\end{definition}

\begin{proposition}
The three V-hyperbolas intersect at the centers of the two Soddy circles, i.e., $X_{175}$ and $X_{176}$, respectively.
\end{proposition}

\begin{proof}
Assume that 
$a>c>b$ as in \cref{fig:v-hyps}. Then 
$r_a<r_c<r_b$.  $\mathcal{C}_a$
and $\mathcal{C}_b$ are two circles  centered at $A$ and $B$ and of radii $r_a<r_b$, which are externally tangent at $C_2$. The locus of the centers of the circles that are externally tangent to both $\mathcal{C}_a$ and $\mathcal{C}_b$ is the branch of the hyperbola with foci on $A$ and $B$, that passes through their tangency point $C_2$. The other branch contains the centers of the circles that are internally tangent (i.e., contain both). The internal Soddy circle is externally tangent to the three circles $\mathcal{C}_a$, $\mathcal{C}_b$, and $\mathcal{C}_c$; hence its center  is necessarily the intersection of the three branches of hyperbolas passing through $A_2,B_2,C_2,$ the vertices of the intouch triangle.
Since $r_a<r_c<r_b$, we may specify those branches as $\H_a^{+}=\{P: |PB|-|PC|=r_b-r_c\}$, $\H_b^{+}=\{P: |PC|-|PA|=r_c-r_a\}$, and  $\H_c^{+}=\{P: |PA|-|PB|=r_a-r_b\}$. 

Thus, if a point 
$P\in \H_a^{+}\bigcap \H_c^{+}$
then 
    $PA-PC=r_a-r_c$ 
  hence it is on $\H_b^{+}$ as well. Since 
  $r_c+r_a=b$ and $r_b+r_a=c$,
  $P$ verifies the equal-detour definition of $X_{176}$.

The points on the other branches contain centers of circles that are internally tangent to the other two;
therefore, if two branches, say
$\H_a^{-} $ and $\H_b^{-}$
have a common point $P$, then, 
as above, $P$ is also on the third branch and is the (unique) center of an external Soddy circle, that contains  $\mathcal{C}_a$, $\mathcal{C}_b$, and $\mathcal{C}_c$. In this case, $P$ verifies the isoperimetric definition of $X_{175}$.

In contrast, if 
$\H_a^{-}$ and $\H_b^{-}$ do not intersect, then there will be no ``negative branch'' intersection. In this case, the three positive branches will intersect in two distinct points: the centers of the inner and outer Soddy circles. Note that each pair 
$(\H_a,\H_b)$, $(\H_b,\H_c)$, and $(\H_c,\H_a)$ have one common focus
$C$, $A$, $B$ respectively; hence they necessarily have four (real) intersections.
This guarantees the existence of both detour and isoperimetric points.
\end{proof}


\begin{figure}
\centering
\includegraphics[width=.6\textwidth]{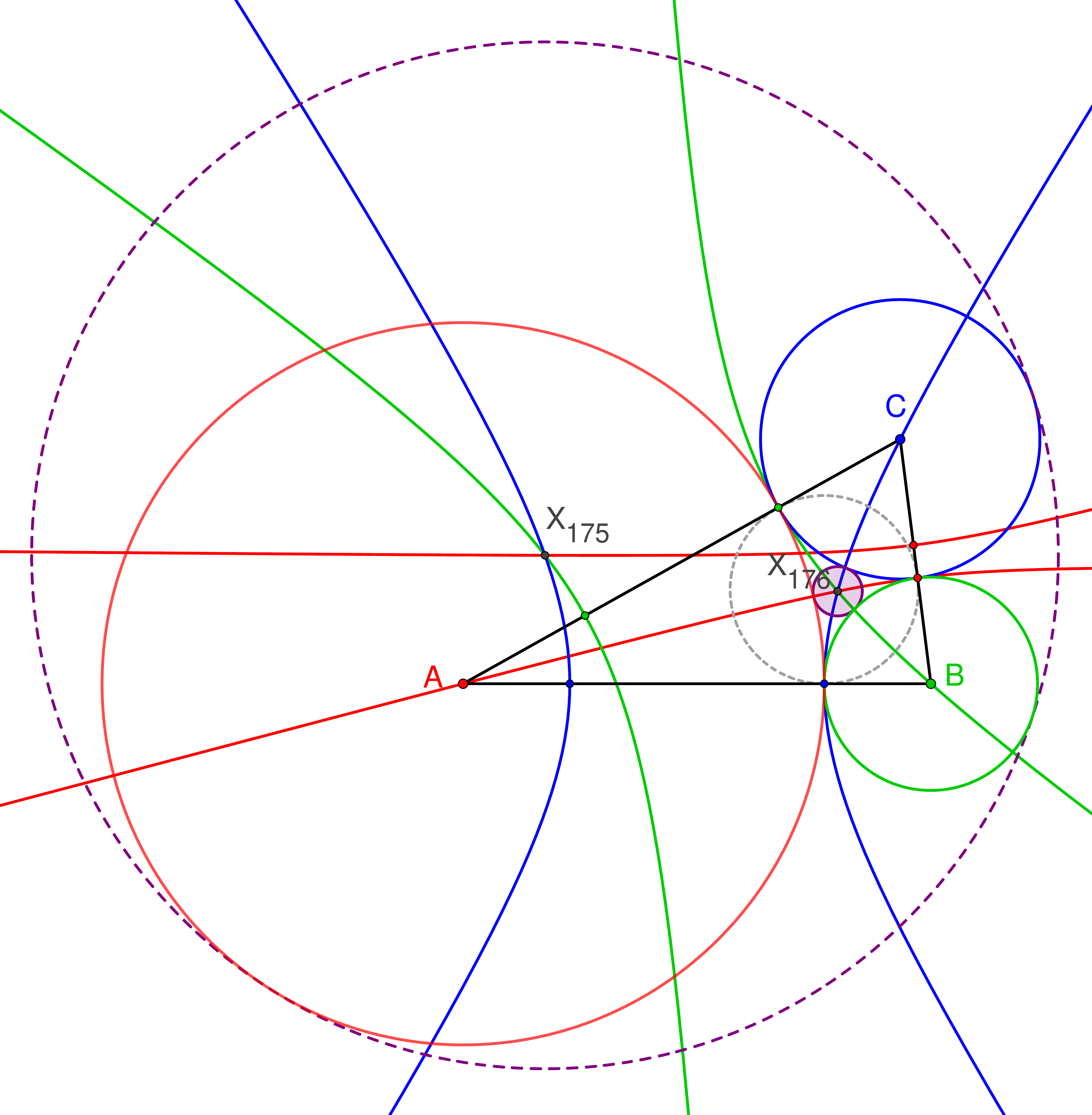}
\caption{A construction found in \cite[Sec. 12.4, p. 148]{yiu2001-intro}: 3 mutually-tangent circles (red, green, blue) of $\triangle ABC$ touch at the contact points of the incircle (dashed gray) with the sides. In turn, these coincide with a vertex of each the 3 V-hyperbolas. Notice the latter intersect at the centers $X_{176}$ and $X_{175}$ of the inner (shaded purple), and outer (dashed purple) Soddy circles, respectively.}
\label{fig:soddy1}
\end{figure}

\begin{figure}
\centering
\begin{subfigure}[m]{0.48\textwidth}
\centering
\includegraphics[trim=0 0 0 0,clip,width=\textwidth]{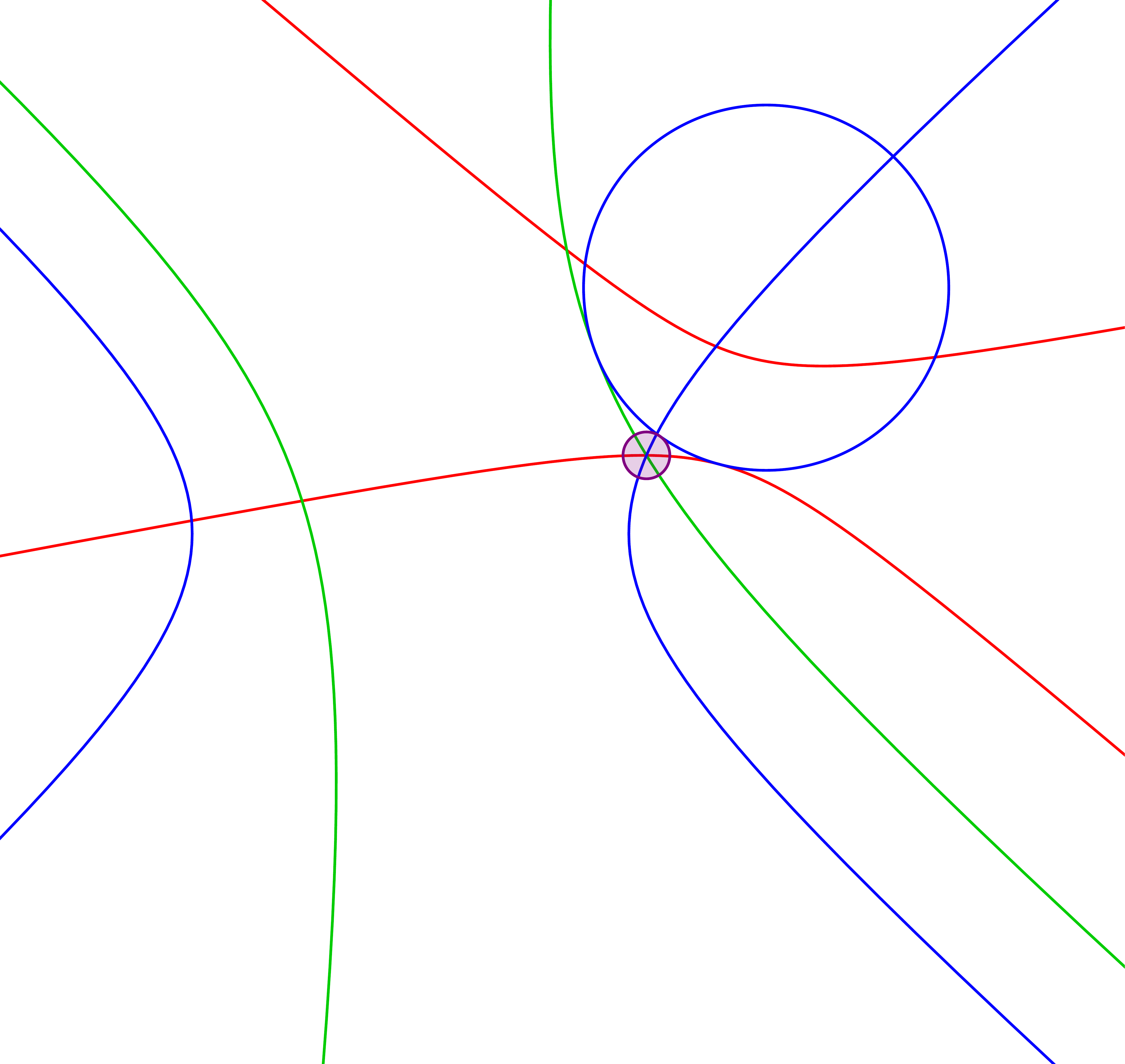}
\end{subfigure}
\rulesep
\begin{subfigure}[m]{0.48\textwidth}
\centering
\includegraphics[trim=0 0 0 0,clip,width=\textwidth]{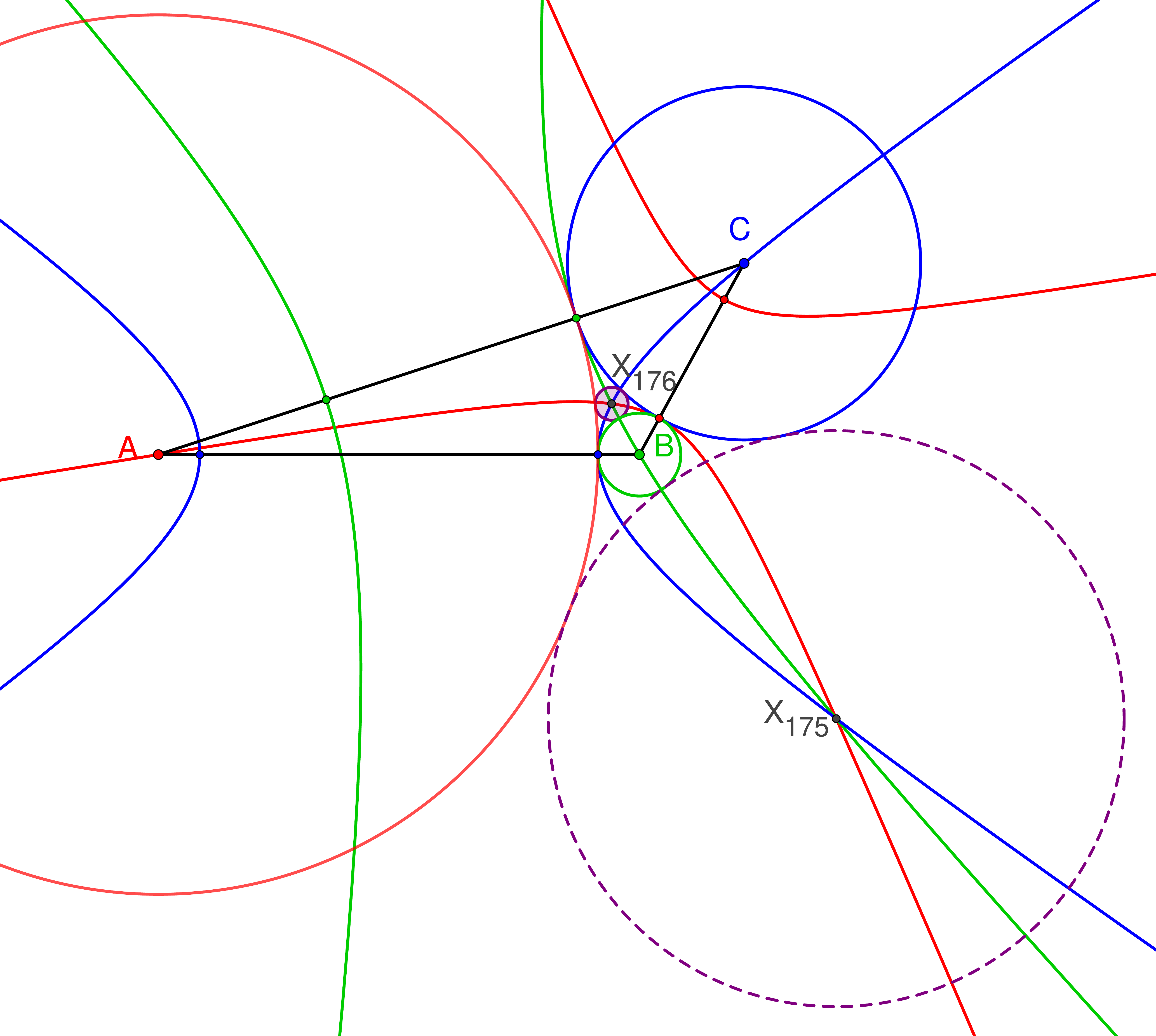}
\end{subfigure}
\caption{Two cases of $\triangle ABC$ such that the external Soddy circle (dashed purple) is: (left) a straight line ( $\sum\tan(\theta_i)=2$), and (right) does not contain the three kissing circles. Notice that in both cases the three V-hyperbolas (red, green and blue) intersect at the center $X_{176}$ of the inner Soddy circle (shaded purple), interior to the triangle. In the first case their second intersection is at infinity (in the direction perpendicular to the Soddy line), while in the second case they intersect along the same branches where their $X_{176}$ intersection lies.}
\label{fig:soddy3}
\end{figure}

\noindent Referring to 
\cref{fig:vhyps-vells}:

\begin{proposition}
The $\H_a$ V-hyperbola passes through the intersections $A'$ and $A''$ of V-ellipses $\E_b$, and $\E_c$. The same holds for $\H_b,\H_c$, cyclically.
\end{proposition}

\begin{proof}
Referring to
\cref{fig:vhyps-vells},
let $A''$ denote an intersection of $\E_b$ with $\E_c$. Then:
\[|A''A|+ |A'' B|=|CA|+|CB|,\;\;\; 
|A''A|+ |A'' C|=|BA|+|BC|.\]
Subtracting, $|A'' C|- |A'' B|=|BA|-|CA|$, meaning that 
$A''$ lies on the branch of hyperbola $\H_a$ not containing  $A$.
\end{proof}

Still referring to \cref{fig:vhyps-vells}, let $A_1$ and $A_2$ denote the two 2-branch intersections between $\H_b$ and $\H_c$, define $B_1,B_2$ and $C_1,C_2$ cyclically.  

\begin{proposition}
The three lines $A_1 A_2$, $B_1 B_2$ and $C_1 C_2$ concur at the Nagel point $X_8$ of $\triangle ABC$. 
\end{proposition}
\begin{proof}
Let $a,b,c$ denote the sidelines. The barycentrics of $A_1$ are given by:
\begin{align*}
    A_1: \left[ \right. &(L-2a) (3 a^2 + 2 a b - b^2 + 2 a c - 2 b c - c^2 + 2 (b - c) \gamma), \\
   & (L-2c) (3 a - 3 b + c) L + 2 (a^2 - c^2 - 2 b^2 + a b - b c ) \gamma, \\
    & (L-2b) (3 a + b - 3 c)L  - 2 (a^2 - b^2 - 2 c^2 + a c + b c ) \gamma \left. \right  ]
\end{align*}
where $L=a+b+c$ and $\gamma = 3 a^2 + 2 a b - b^2 + 2 a c - 2 b c - c^2$. The barycentrics of $A_2$ are obtained by replacing $\gamma$ with $-\gamma$. The barycentrics of points on $A_1 A_2$ satisfy:
\[ (-2 a^2 + a b + b^2 + a c - 2 b c + c^2) x + (a - c) (L-2a) y + (a - b) (L-2a) z = 0 \]
and cyclically for $B_1 B_2$ and $C_1 C_2$.
It can be shown that the 3 lines pass through $X_8$, whose barycentric coordinates are $[b+c-a,a+c-b, a+b-c]$, see \cite{etc}.
\end{proof}

\begin{figure}
\centering
\includegraphics[trim=200 0 0 0,clip,width=.8\textwidth,frame]{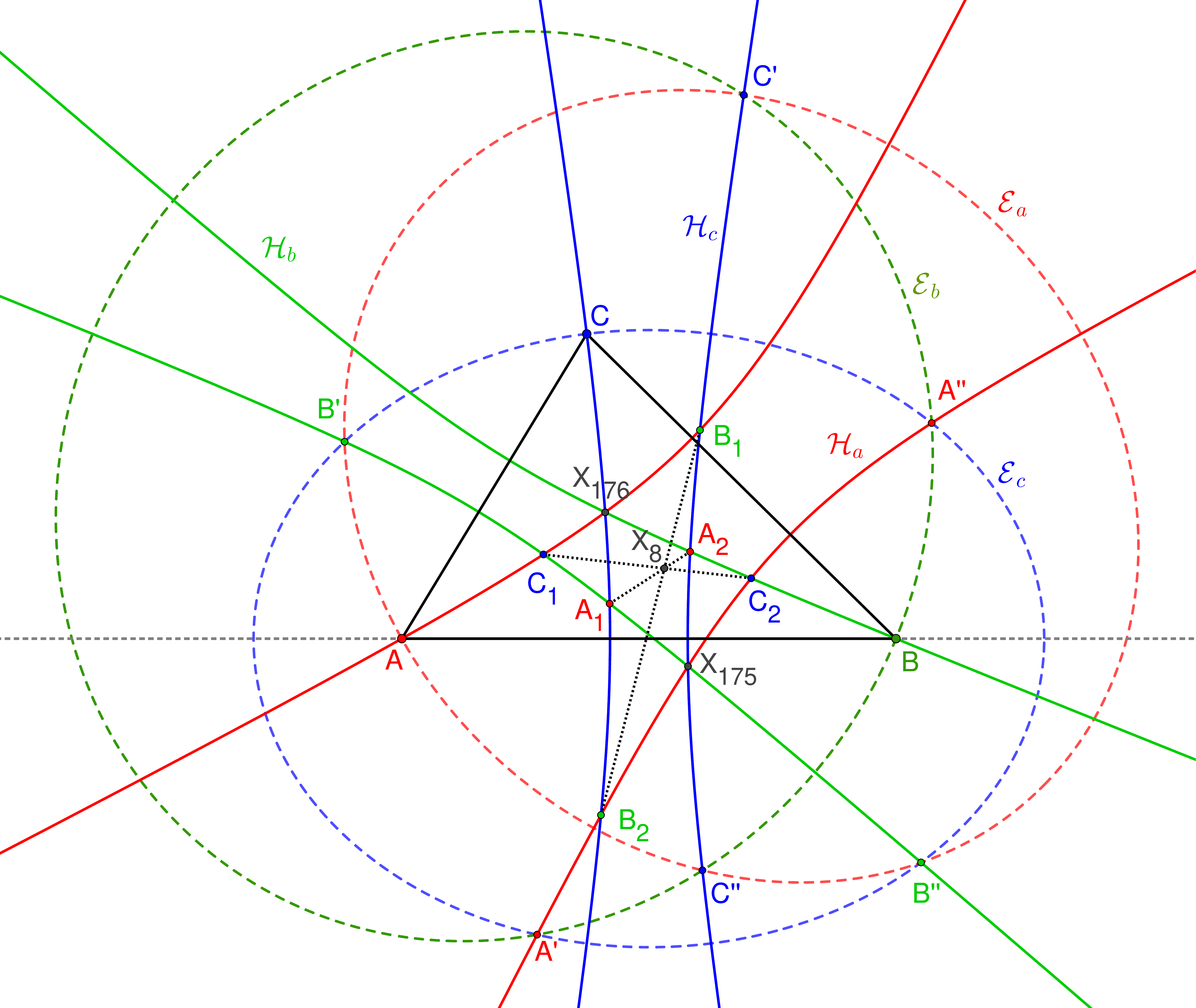}
\caption{Each V-hyperbola passes through the two intersections between pairs of V-ellipses. As an example, consider $A''$, common to $\H_a$, $\E_b$, and $\E_c$. Also shown is the fact that the 3 segments $A_1A_2$, $B_1 B_2$, and $C_1C_2$ connecting opposing 2-branch intersections of the 3 V-hyperbolas concur at $X_8$.}
\label{fig:vhyps-vells}
\end{figure}

Referring to \cref{fig:vcubic}:

\begin{remark}
On side $BC$ there lie the 2 vertices of $\E_a$ and the 2 of $\H_a$. Consider the degenerate cubic which is the union of the sidelines of $\triangle ABC$. It is a 15-point cubic since it passes through (i) the three vertices of the triangle, (ii) the 6 vertices of the V-ellipses, and (iii) the 6 vertices of the V-hyperbolas.
\end{remark}

\begin{figure}
\centering
\includegraphics[trim=0 0 0 0,clip,width=.6\textwidth]{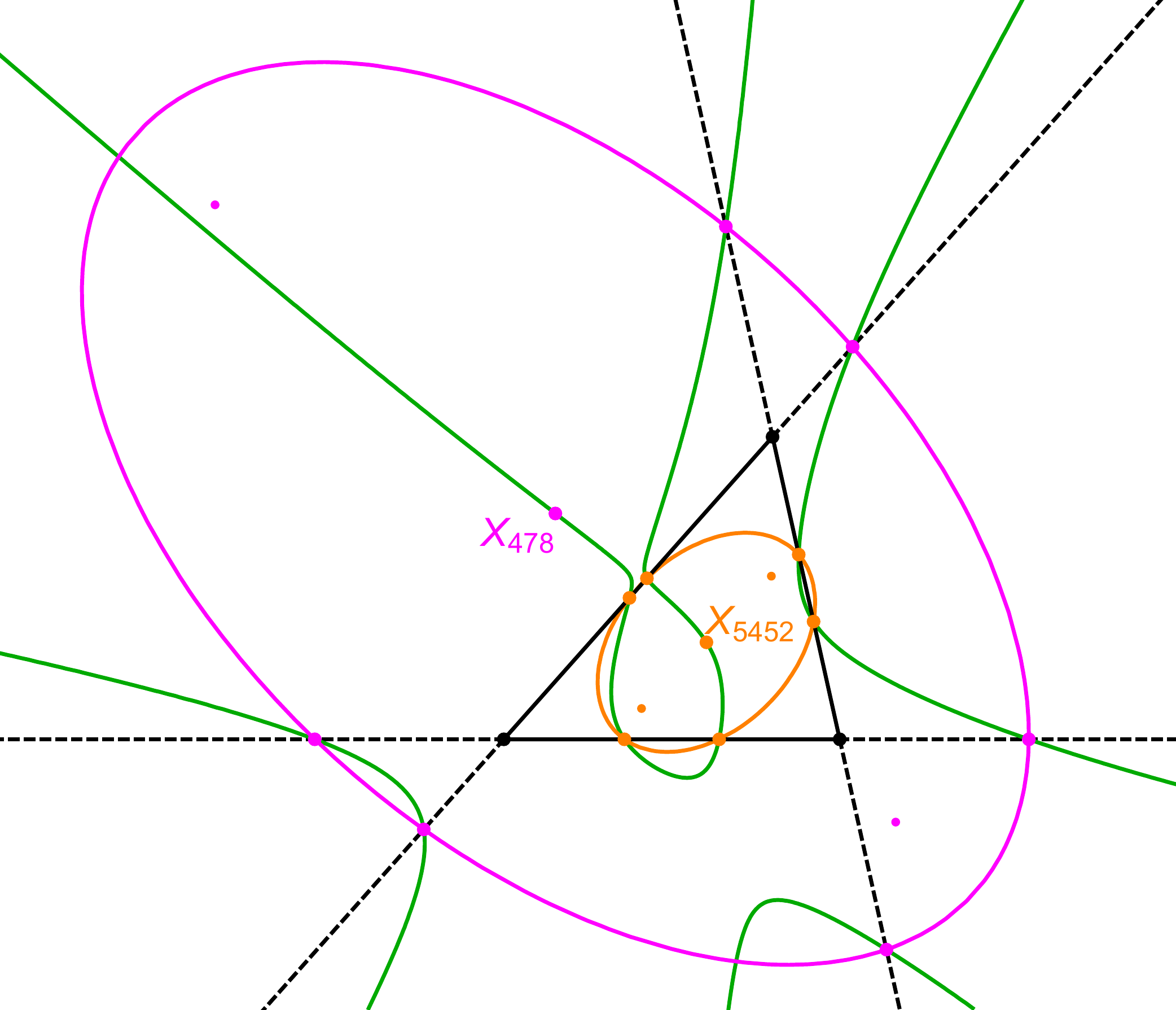}
\caption{Given a triangle (black), the union of the 3 sidelines (dashed black) can be regarded as a 15-point degenerate cubic. It passes through (i) the triangle vertices, (ii) the 6 points on the Yiu conic (magenta), and (iii) the 6 points on the Privalov conic (orange). Just for fun, also shown are branches of the 14-point quartic (green) that passes through the 12 points on the Yiu+Privalov as well as their centers $X_{478}$ and $X_{5452}$, respectively.}
\label{fig:vcubic}
\end{figure}

Referring to \cref{fig:half-tangs}:

\begin{proposition}
When the external Soddy circle degenerates to a line, the three circles whose diameters are the sides of $\triangle{ABC}$ are also tangent to it.
\end{proposition}
\begin{proof}

Let $A_T,$ $B_T$ be the tangency points of the degenerate Soddy circle $\L$ (a line) with circles $\C_a,\C_b$,
and let $T$ be the tangency point between the latter two.
The perpendicular dropped from $T$ onto line $AB$ meets $\L$ at $M$. Then, owing to properties of tangents from a points to a circle, $|M T|=|M A_T|=|M B_T|$.  Since $B_T,M,A_T$ are collinear, then $\angle{A_T T B_T}=90^{\circ}$. On the other hand, since $MT$ and $MA_T$
are tangents from $M$ to $\C_a$, $MA\perp TA_T$
and similarly, $MB\perp TB_T$. $\angle{A M B}=90^{\circ}$. Hence, if $O$ is the midpoint of $AB$ then $OM=AO=OB$. Finally, the quadrilateral $[A B B_T A_T]$
is a trapezium ($AA_T$, $BB_T$ are perpendicular to $\L$) and $OM$ is its mid-base. Hence $OM$ is also perpendicular to $\L$ at the midpoint $M$ of $A_T B_T$. Therefore the circle of diameter $AB$ is tangent to $\L$ at $M$, and so on cyclically for  $(\C_b,\C_c)$ and $(\C_a,\C_c)$.
\end{proof}

\section{A triad of P-hyperbolas}
\label{sec:p-hyp}
We now extend V-hyperbolas to a trio with respect to a point $P$. Referring to \cref{fig:p-hyp}:

\begin{definition}[P-hyperbolas]
A triad of P-hyperbolas $\H_a^*,\H_b^*,\H_c^*$ with respect to $\triangle ABC$ have foci on $(B,C)$, $(C,A)$, $(A,B)$ and pass through a given point $P$.
\end{definition}

\begin{figure}
\centering
\includegraphics[width=.8\textwidth,frame]{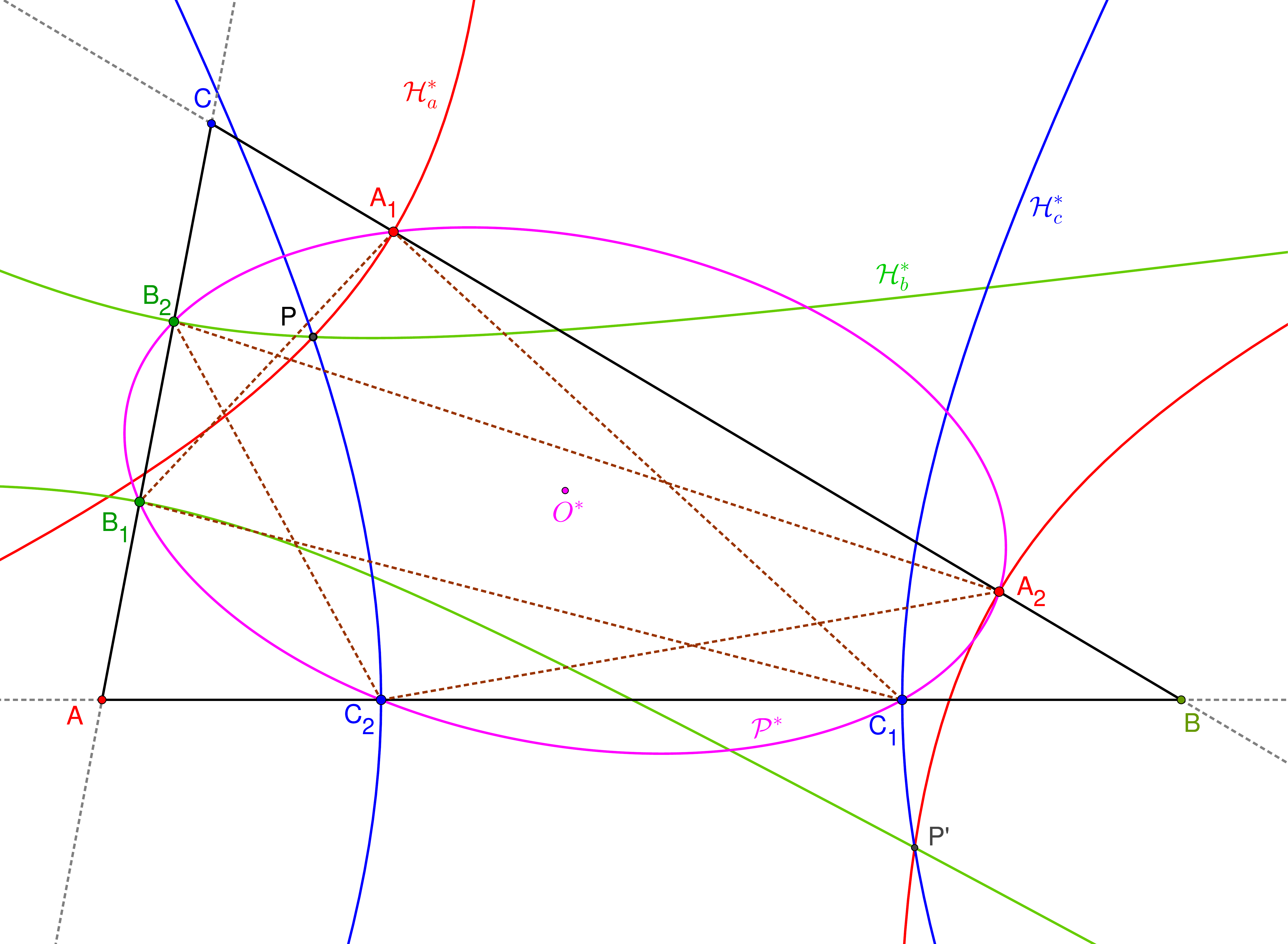}
\caption{Given a point $P$, (i) the triad of P-hyperbolas $\H_a^*,\H_b^*,\H_c^*$ has a second common point $P'$; (ii) their vertices $A_1,A_2,B_1,B_2,C_1,C_2$ lie on a conic $\P^*$ (magenta), not necessarily an ellipse; (iii) $\triangle A_1 B_1 C_1$ has the same area as $\triangle A_2 B_2 C_2$.}
\label{fig:p-hyp}
\end{figure}

Still referring to \cref{fig:p-hyp}:

\begin{proposition}
Besides $P$, the triad of P-hyperbolas meets at a second real point $P'$.
\end{proposition}

\begin{proof}
Let $\H_a^{*+}$, $\H_b^{*+}$, $\H_c^{*+}$ denote the three branches that pass through $P$. We need to prove that the other three branches
$\H_a^{*-}$, $\H_b^{*-}$, $\H_c^{*-}$ also meet at some point.

First, let us show that  two other branches $\H_b^{*-}$ and $\H_c^{*-}$ must intersect. To prove it,  we perform a polar dual with respect to a circle centered at their common focus $A$, as shown in 
\cref{fig:polar-proof}.
 The polar dual of each hyperbola will be a circle, whose diameter is delimited by the inverses of hyperbola vertices. By polarity, the intersection points of the original hyperbolas are sent to the common tangents of their reciprocal circles and vice-versa; since, by hypothesis, $\H_b^{*}$ and $\H_c^{*}$ intersect at a point $P$, these reciprocal circles admit (at least) one common tangent. Hence, they are either externally tangent or secant. Therefore, these circles admit at least two  common tangents. One of these tangents is precisely the polar of $P$; the other one, passing through the same homothety center, 
 is the polar of a point $P'$
which is the intersection of the other two branches,  $\H_b^{*-}$ and  $\H_c^{*-}$. Similarly, branches $\H_a^{*-}$ and  $\H_c^{*-}$ also intersect. Now,
as in \cref{fig:p-hyp},  if a point  $P'\in \H_b^{*-}\bigcap \H_c^{*-}$, then it satisfies  %
$P'C-P'A=B_1B_2$
and $P'A-P'B=C_1 C_2$. Hence, by adding these two relations, we obtain 
$P'C-P'B=B_1B_2+C_1C_2.$ Nevertheless, by hypothesis
$P$ is the common point of three branches: $\H_a^{*+}, \H_b^{*+},\H_c^{*+}.$ Then three similar relations can be written for $P:$
$PA-PC=B_1B_2$, $PB-PA=C_1C_2$, and $PB-PC=A_1A_2.$
By adding the first two, we obtain 
$PB-PC=B_1B_2+C_1C_2,$ hence $B_1B_2+C_1C_2=A_1A_2.$
The later relation ensures that $P'C-P'B=A_1A_2,$ hence $P'\in \H_a^{*-}$
finishing the proof.

\end{proof}

\begin{figure}
\centering
\includegraphics[trim=0 0 0 0,width=.6\textwidth]{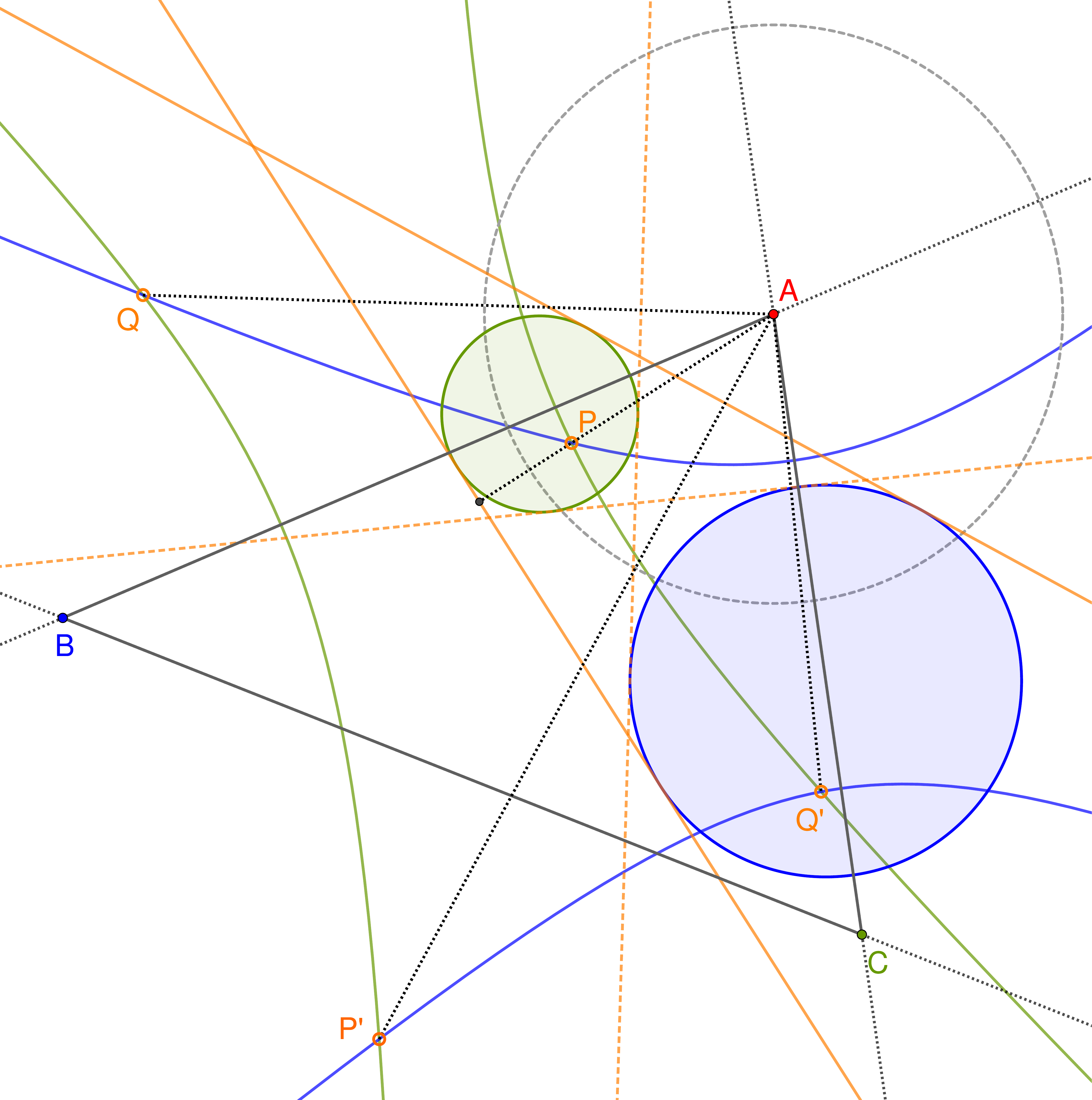}
\caption{Two P-hyperbolas (green, blue) sharing a focus at $A$ are shown as well as their reciprocals (shaded green and blue circles) with respect to an inversion circle centered at $A$ (dashed black). The branches closer (resp. further) to $A$ always intersect; their intersection points $P$ and $P'$ are the poles of the common external tangents to their reciprocal circles.  When these circles are disjoint (as in the figure), the poles of the internal common tangents are intersections between alternate branches.}
\label{fig:polar-proof}
\end{figure}

\begin{proposition}
The 6 vertices of the 3 P-hyperbolas lie on a conic $\P^*$.
\end{proposition}

\begin{proof}
Referring to \cref{fig:p-hyp}, by definition, the center of the P-hyperbola $\H_a^*$ is at the midpoint of $BC$, and so on cyclically. Hence:
\begin{equation}
|A_1 C|=|A_2B|=x,\;\;\;|B_1 A|=|B_2 C|=y,\;\;\;|BC_1|=|AC_2|=z.
\label{eq:p-Carnot}
\end{equation}
We obtain the claim using Carnot's theorem.
\end{proof}

Recall the classic result that for any triangle, the intouch and extouch triangles have the same area (we saw this in \cref{prop:extouch-intouch} in the context of V-hyperbolas). The analogous result for P-hyperbolas still holds:

\begin{proposition}
Let $A_1,A_2$ denote the vertices of $\H_a^*$, and $B_1,B_2$, $C_1,C_2$, those of $\H_b^*$ and $\H_c*$, respectively. Then $\triangle A_1 B_1 C_1$ and $\triangle A_2 B_2 C_2$ have the same area.
\end{proposition}

\begin{proof} 
This is again a consequence of  \cref{eq:p-Carnot}. Specifically, let $\vec a=  \vv{BC}$, $\vec b=  \vv{CA}$,
$\vec c=  \vv{AB}$. Let
$\alpha$, $\beta$, and $\gamma$ be such that:

\[    \vv{BA_1}=\alpha\vec{a},   \;\;\;    \vv{CB_1}=\beta\vec{b}, \;\;\;
\vv{AC_1}=\gamma\vec{c}; \]

In order to prove that $S_{A_1B_1C_1}=S_{A_2B_2C_2}$, we simply show that they represent the same fraction of $S=S_{ABC}$.
In fact, $S_{A_1B_1C_1}=S-[S_A+S_B+S_C]$, 
where $S_A=S_{AB_1C_1}$, $S_B=S_{BA_1C_1}$, and $S_C=S_{CA_1B_1}$.
A direct computation yields:
\[ S_A=S_{AB_1C_1}=\frac{1}{2}\|\vv{AB_1}\,\times\,\vv{AC_1}\|=
\frac{1}{2}\|{(1-\beta)\vec{b}}\times(\gamma \vec{c})\|=\gamma(1-\beta)S \]
Cyclically, $S_B=\alpha(1-\gamma)S$, and $S_C=\beta(1-\alpha)S$. Therefore:
\[S_{A_1 B_1 C_1}=S-\big[S_A+S_B+S_C \big]= 
S\big[1- \gamma(1-\beta)- \alpha(1-\gamma)-\beta(1-\alpha)\big]\]   
Similarly:
\[S_{A_2 B_2 C_2}=S-\big[S'_A+S'_B+S'_C]    \]
where $S'_A=S_{A B_2 C_2},$ $S'_B=S_{B A_2 C_2},$ $S'_C=S_{C A_2 B_2}$.
Then:
\[S'_A=S_{A B_2 C_2}=\frac{1}{2}\|\vv{A B_2}\,\times\,\vv{A C_2}\|=
\frac{1}{2}\|{\beta \vec{b}}\times((1-\gamma) \vec{c})\|=\beta(1-\gamma)S.
\] and cyclically for $S'_B$, $S'_C.$ Thus,
the area of $S_{A_2B_2C_2}$ can be computed as $S_{A_1 B_1 C_1}$, where $\alpha,\beta,\gamma$ are replaced at each occurrence by $(1-\alpha)$, $(1-\beta)$, $(1-\gamma)$.
Thus:
\[S_{A_2 B_2 C_2}=S-\big[S'_A+S'_B+S'_C \big]= 
S\big[1- (1-\gamma)\beta- (1-\alpha)\gamma-\alpha(1-\beta)\big]\]
Hence, the two areas are equal. \end{proof}

Referring to \cref{fig:phyp-circular}:

\begin{proposition}
Given a $\triangle ABC$ there is a unique pair of distinct points $P^*$ and $Q^*$ such that the 6-point conic $\P^*$ is a circle. These are a pair of common intersections of the triad of P-hyperbolas. It can be shown their barycentrics satisfy:
{\small
\begin{align*}
& \left[(c^2-\lambda_c^2) (-a^2+b^2+\lambda_a^2-\lambda_b^2)\right]^2+\\
& \left[(b^2-\lambda_b^2)(-a^2+c^2+\lambda_a^2-\lambda_c^2)\right]^2+\\
&\left[(a^2-\lambda_a^2)(-b^2+c^2+\lambda_b^2-\lambda_c^2)\right]^2 = 0  
\end{align*}
}
where  $\lambda_a = |PB|-|PC|$,  $\lambda_b=|PC|-|PA|$, and $\lambda_c=|PA|-|PB|$.
\end{proposition}

\begin{figure}
\centering
\includegraphics[width=.6\textwidth,frame]{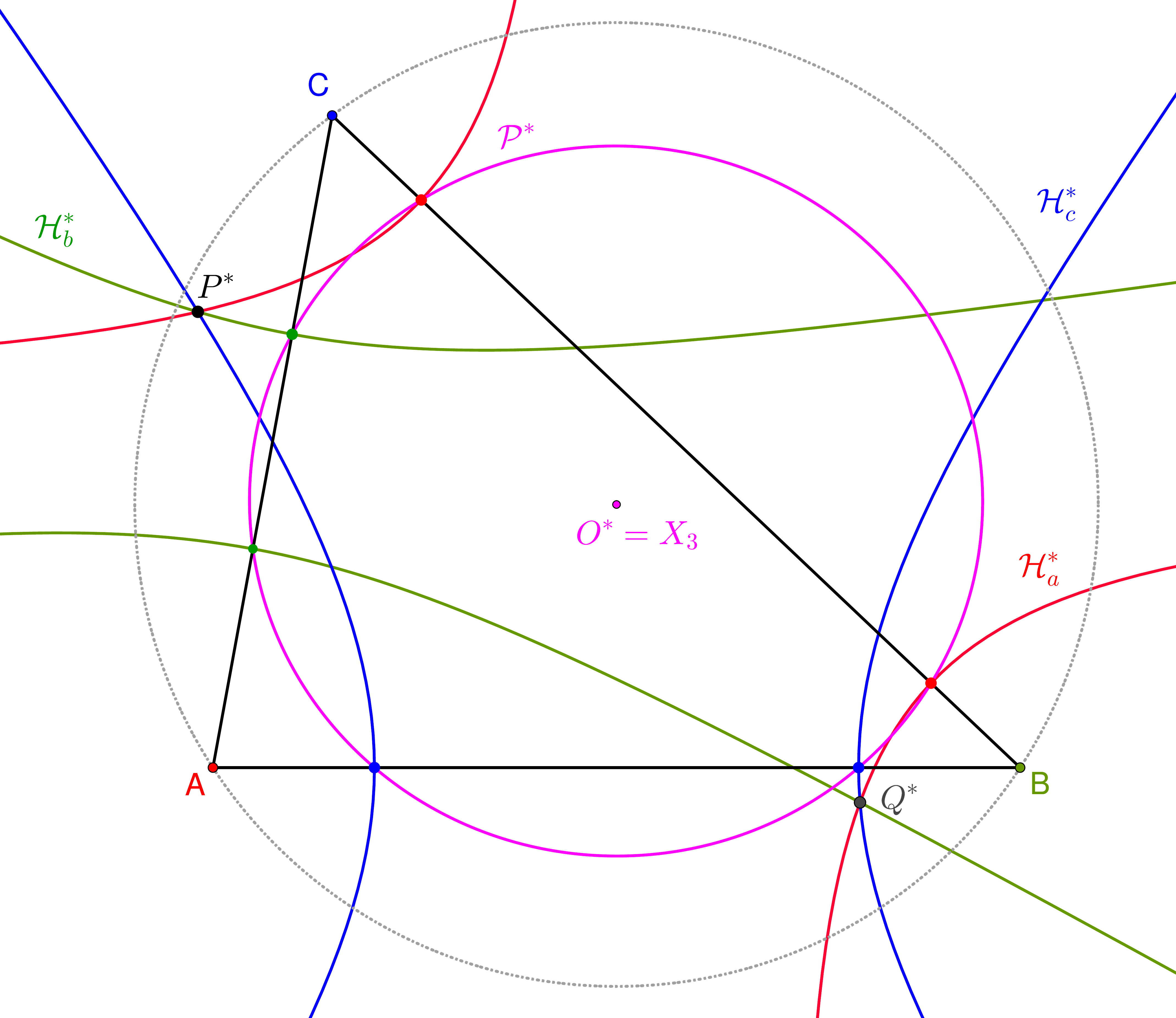}
\caption{Given $\triangle ABC$, there is a pair $P^*$ and $Q^*$ such that $\P^*$ (magenta) is a circle. Furthermore the latter is concentric with the circumcircle (dashed black) of $\triangle ABC$}
\label{fig:phyp-circular}
\end{figure}

\begin{definition}[reflection triangle]
The reflections $A',B',C'$ of a point $Q$ on the sides of $\triangle ABC$ are the vertices of the $Q$-reflection triangle.
\end{definition}

Surprisingly, we can construct a triangle such that the vertices of the 6 P-hyperbolas lie on a circle. In \cite{etc}, center $X_{55}$ is the internal center of similitude of the incircle and circumcircle. 

Referring to \cref{fig:x55}, experimental evidence supports the following ``needle in a haystack'' phenomenon:

\begin{conjecture}
Let $T'$ be the $X_{55}$-reflection triangle
of a reference triangle $T$. The 6 vertices of the P-hyperbolas of $T'$ passing through $X_{55}$-of-T lie on a circle, concentric with the circumcircle of $T'$ which coincides with $X_7$-of-T. 
\label{conj:x55}
\end{conjecture}

\begin{figure}
\centering
\includegraphics[width=.8\textwidth,frame]{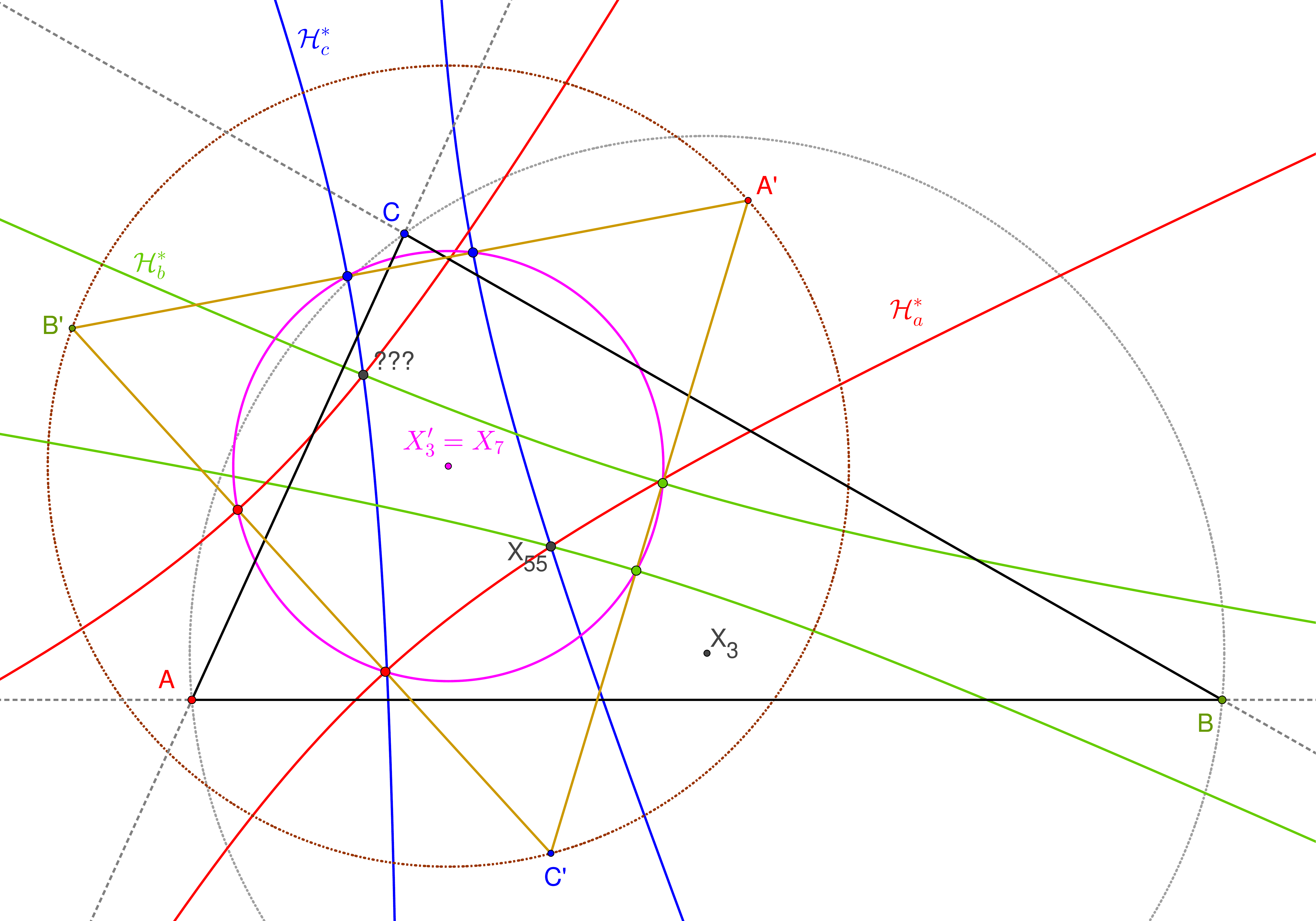}
\caption{The vertices of P-hyperbolas $\H_a^*$, $\H_b^*$, and $\H_c^*$ (red, green, blue) passing through $X_{55}$ of $\triangle ABC$,  with foci on pairs of vertices of the $X_{55}$-reflection triangle $\triangle A'B'C'$ (gold),  lie on a circle (magenta), concentric with the circumcircle (dotted brown) of $\triangle A'B'C'$, whose circumcenter is the Gergonne point $X_7$ of the reference. Note that said 3 hyperbolas meet at a second mystery point ``???''.}
\label{fig:x55}
\end{figure}

Referring to \cref{fig:phyps-pells}:

\begin{proposition}
The $\H_a^*$ P-hyperbola passes through the non-P intersection $A'$ between P-ellipses $\E_b^*$, and $\E_c^*$. The same holds for $\H_b^*,\H_c^*$, cyclically.
\end{proposition}

\begin{proof}
Referring to figure 
\cref{fig:phyps-pells}, let $A'\in \E_b^*$;
then $A'A+A'C=PA+PC$. If 
$A'$ is also contained in $\E_c^*$, then
 $A'A+A'B=PA+PB$. Subtracting
 $A'B-A'C=PB-PC$, i.e., both $A'$ and $P$ lie on the same branch of $\in \H_a^*$.
\end{proof}

\begin{figure}
\centering
\includegraphics[width=.7\textwidth,frame]{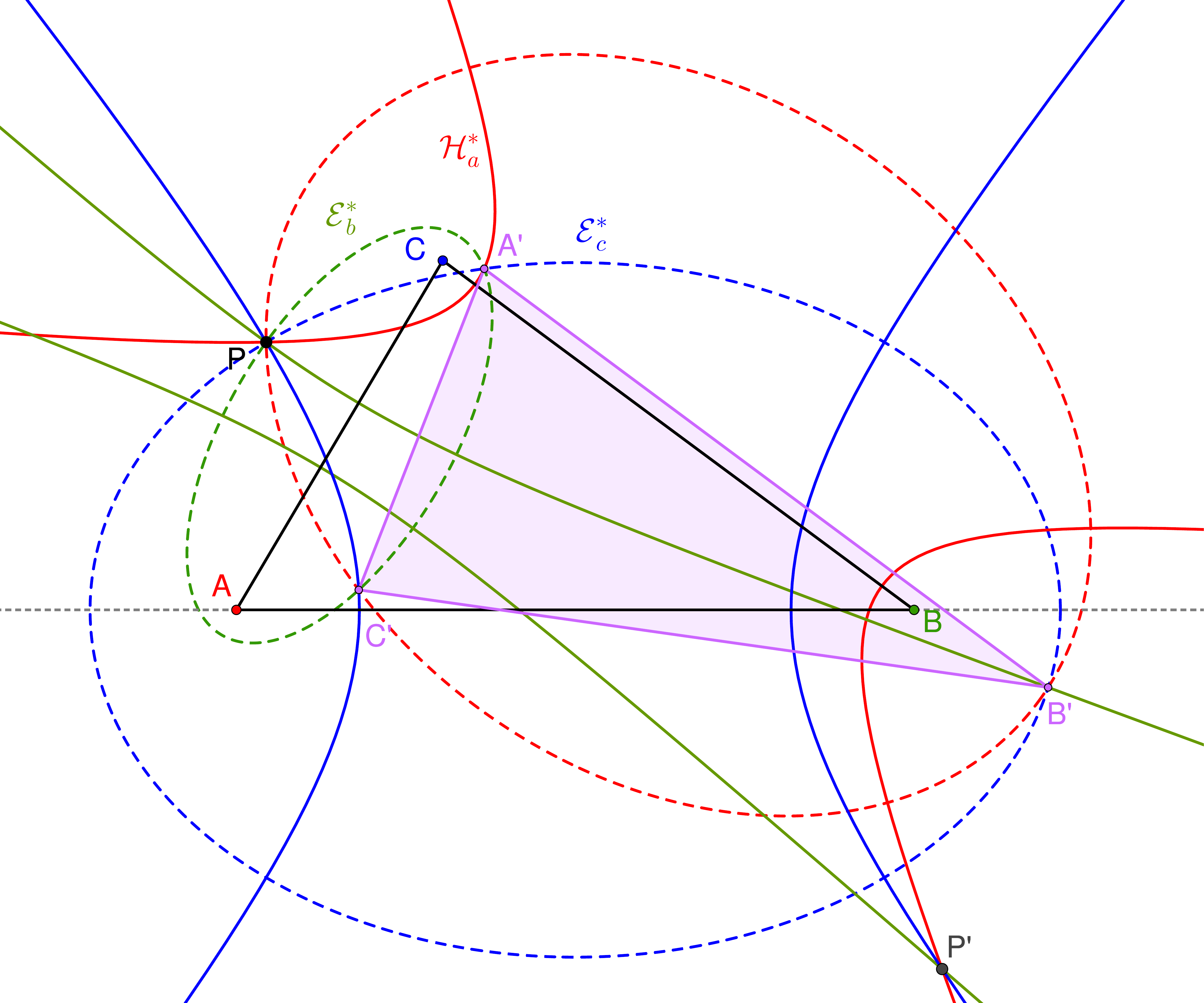}
\caption{The 3 P-hyperbolas (solid red, green, blue) also pass through the 3 non-P intersections $A',B',C'$ between pairs of P-ellipses, e.g., $\H_a^*$ passes through the intersection $A'$ of $\E_b^*$ and $\E_c^*$. The triangle with vertices on $A',B',C'$ is shown (purple).}
\label{fig:phyps-pells}
\end{figure}

\begin{figure}
\centering
\includegraphics[width=\textwidth]{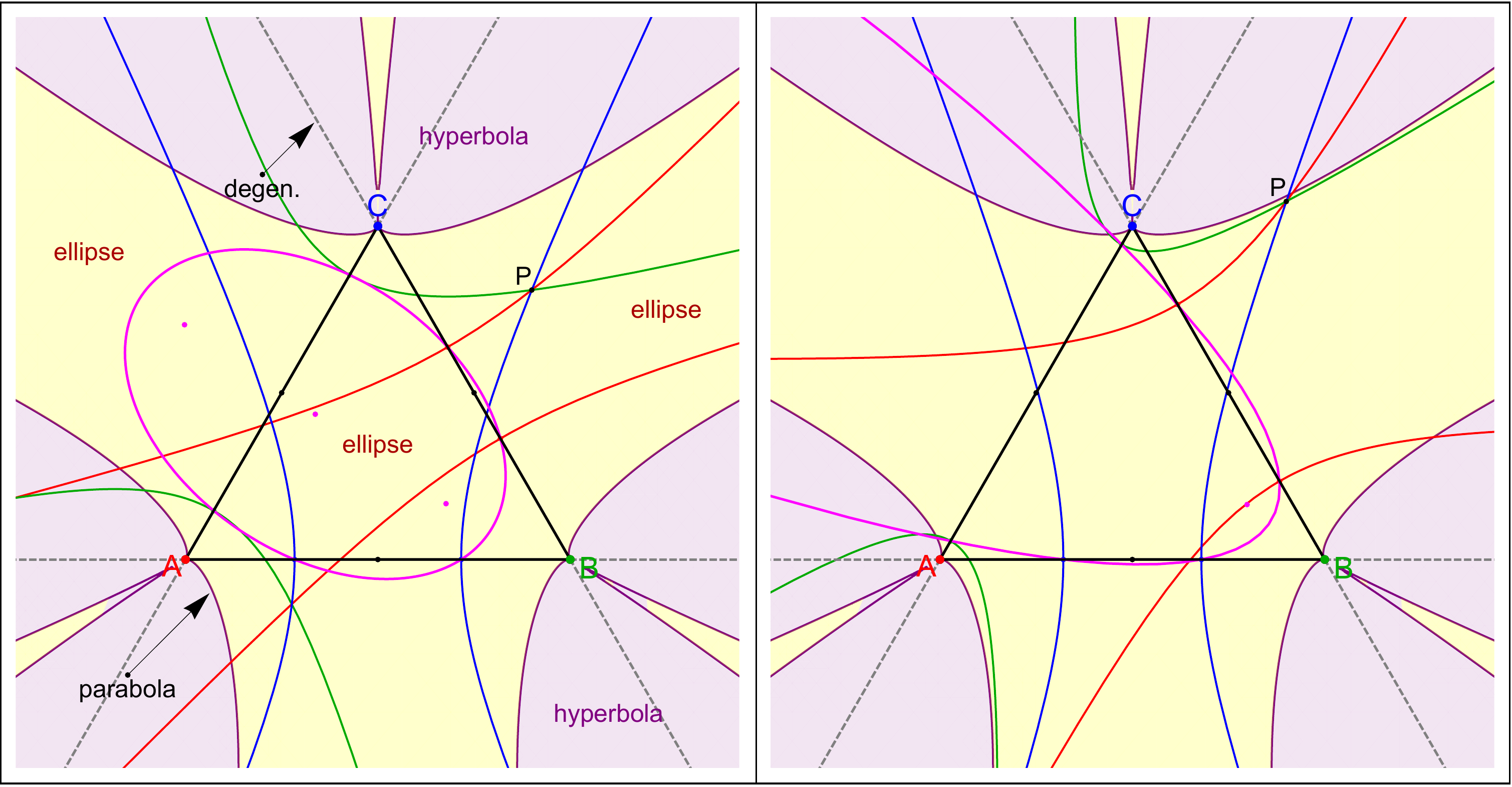}
\caption{The 6-point conic $P^*$ (magenta) through the vertices of P-hyperbolas (red, green, blue) is an ellipse if $P$ lies in the yellow (resp.) purple region. It is degenerate if $P$ is on any sideline (dashed black). In the left (resp. right) $P$ is in the yellow region (at the interface) and therefore $P^*$ (magenta) is an ellipse (resp. parabola).}
\label{fig:phyp-zones}
\end{figure}

As before, let $a,b,c$ be the sidelengths, and $\lambda_a$, $\lambda_b$, and $\lambda_c$ as above. As shown in \cref{fig:phyp-zones}, the plane of a $\triangle ABC$ can be split into zone where $\P^*$ is an ellipse, a hyperbola, or a parabola. In particular:

\begin{proposition}
The conic $\P^*$ through the 6 vertices of the P-hyperbolas is a parabola if:
\begin{align*}
&(c^2 \lambda_a^2 \lambda_b^2)^2 + (b^2 \lambda_a^2 \lambda_c^2)^2 + (a^2 \lambda_b^2 \lambda_c^2)^2 +
\mu (\mu (-3 + 4 (\lambda_a^2/a^2 + \lambda_b^2/b^2 + \lambda_c^2/c^2)) + \\
&  2 \lambda_a^2 \lambda_b^2 \lambda_c^2 (6 - \lambda_a^2/a^2 - \lambda_b^2/b^2 - \lambda_c^2/c^2) -
    6  (c^2 \lambda_a^2 \lambda_b^2 + b^2 \lambda_a^2 \lambda_c^2 + a^2 \lambda_b^2 \lambda_c^2))=0
\end{align*}
where $\mu=(a b c)^2$. Furthermore, $\P^*$ is degenerate if $P$ lies on either (infinite extension) of the sidelines of the triangle.
\end{proposition}

\section{Open Questions}
The following questions are posed to the reader:

\begin{compactenum}
    \item \cref{fig:vells-zones}: what is the locus of the focus of the Yiu conic over $C$ along the parabola locus?
    \item \cref{fig:vells-covertices-quad}: prove the 6-point co-vertex conic is always a hyperbola and explain why there are two (3:3 and 5:1) distributions of co-vertices over the branches of the conic.
    \item \cref{fig:circ-anticev}(left): Prove $P^*$ is unique.
    \item \cref{fig:circ-anticev}(right): given $\triangle A'B'C'$ can one always find an inscribed $\triangle ABC$ such that the former is its $X_{3}$-anticevian triangle?
    \item \cref{fig:equi}: how do the zones of 6-vertex conic type deform as one moves $C$ away from the equilateral configuration? What is the locus of the center $O^*$ of the degenrate conic over $P$ on the 3 branches of the inner deltoid? What is the locus of the focus of the conic over the 6 arcs where the conic is a parabola? Prove if a hyperbola, said conic can never be rectangular.
    \item \cref{fig:pells-covertices}: Prove that if $P$ is on the locus, the center $O\dagger$ of the co-vertex conic is on the incircle. What does the locus of $P$ look like if $\triangle ABC$ is not an equilateral? Over $P$ on said locus, what is the locus of $O\dagger$? 
    \item prove \cref{conj:x55}. Provide an expression for the second triple intersection point of the 3 P-hyperbolas.
    \item \cref{fig:vcubic} What are interesting loci for $C$ (with $A,B$ fixed) with respect to properties and/or degeneracies of the 14-point quartic?
    \item \cref{fig:p-hyp}: describe the map $P\to P'$ and/or $P \to O^*$? What is the image of a lattice under it?
    \item \cref{fig:x55}: given $\triangle A'B'C'$ can one always find a $\triangle ABC$ such that the former is its $X_{55}$-reflection triangle? 
\end{compactenum}


\section*{Acknowledgements}
\noindent We would like to thank Arseniy Akopyan and Boris Odehnal for useful discussions. The first author is fellow of CNPq and coordinator of Project PRONEX/ CNPq/ FAPEG 2017 10 26 7000 508.

\appendix

\section*{Long Barycentric Equations}
\label{app:barys}
Here we provide barycentric equations and coordinates (with respect to the reference $\triangle ABC$) of various associated objects. Let $a,b,c$ be the reference's sidelengths. Let $\E_a^\dagger$ denote the ellipse with foci on $B'C'$, and through $A$ of the reference. Note: the long expression below were kept verbatim so as to facilitate copy-paste.

\subsection*{A-ellipse:} The barycentrics $[x,y,z]$ of $\E_a^\dagger$ satisfy:
{\scriptsize
\begin{verbatim}
8*b^4*c^4*(a^2+b^2-c^2)*(a^2-b^2+c^2)*(a^6*b^4-3*a^4*b^6+3*a^2*b^8-b^10+3*a^4*b^4*c^2-
6*a^2*b^6*c^2+3*b^8*c^2+a^6*c^4+3*a^4*b^2*c^4+6*a^2*b^4*c^4-2*b^6*c^4-3*a^4*c^6-
6*a^2*b^2*c^6-2*b^4*c^6+3*a^2*c^8+3*b^2*c^8-c^10)*x^2+4*b^2*c^4*(a^2+b^2-c^2)^2*(a^2-
b^2+c^2)*(a^8*b^2-2*a^6*b^4+2*a^2*b^8-b^10-a^8*c^2-2*a^6*b^2*c^2+10*a^4*b^4*c^2-
10*a^2*b^6*c^2+3*b^8*c^2+4*a^6*c^4+4*a^4*b^2*c^4+10*a^2*b^4*c^4-2*b^6*c^4-6*a^4*c^6-
6*a^2*b^2*c^6-2*b^4*c^6+4*a^2*c^8+3*b^2*c^8-c^10)*x*y+a^2*c^4*(a^2-b^2-c^2)*(a^2+b^2-
c^2)*(a^12-7*a^8*b^4+16*a^6*b^6-21*a^4*b^8+16*a^2*b^10-5*b^12-6*a^10*c^2-4*a^8*b^2*c^2+
12*a^6*b^4*c^2+24*a^4*b^6*c^2-38*a^2*b^8*c^2+12*b^10*c^2+15*a^8*c^4+16*a^6*b^2*c^4+
6*a^4*b^4*c^4+32*a^2*b^6*c^4-5*b^8*c^4-20*a^6*c^6-24*a^4*b^2*c^6-20*a^2*b^4*c^6-8*b^6*c^6+
15*a^4*c^8+16*a^2*b^2*c^8+9*b^4*c^8-6*a^2*c^10-4*b^2*c^10+c^12)*y^2-4*b^4*c^2*(a^2+b^2-
c^2)*(a^2-b^2+c^2)^2*(a^8*b^2-4*a^6*b^4+6*a^4*b^6-4*a^2*b^8+b^10-a^8*c^2+2*a^6*b^2*c^2-
4*a^4*b^4*c^2+6*a^2*b^6*c^2-3*b^8*c^2+2*a^6*c^4-10*a^4*b^2*c^4-10*a^2*b^4*c^4+2*b^6*c^4+
10*a^2*b^2*c^6+2*b^4*c^6-2*a^2*c^8-3*b^2*c^8+c^10)*x*z-2*a^2*b^2*c^2*(a^2-b^2-c^2)*(a^2+
b^2-c^2)*(a^2-b^2+c^2)*(a^10-3*a^8*b^2+2*a^6*b^4+2*a^4*b^6-3*a^2*b^8+b^10-3*a^8*c^2+
8*a^6*b^2*c^2-14*a^4*b^4*c^2+16*a^2*b^6*c^2-7*b^8*c^2+2*a^6*c^4-14*a^4*b^2*c^4-26*a^2*b^4*c^4+
6*b^6*c^4+2*a^4*c^6+16*a^2*b^2*c^6+6*b^4*c^6-3*a^2*c^8-7*b^2*c^8+c^10)*y*z+a^2*b^4*(a^2-b^2-
c^2)*(a^2-b^2+c^2)*(a^12-6*a^10*b^2+15*a^8*b^4-20*a^6*b^6+15*a^4*b^8-6*a^2*b^10+b^12-
4*a^8*b^2*c^2+16*a^6*b^4*c^2-24*a^4*b^6*c^2+16*a^2*b^8*c^2-4*b^10*c^2-7*a^8*c^4+12*a^6*b^2*c^4+
6*a^4*b^4*c^4-20*a^2*b^6*c^4+9*b^8*c^4+16*a^6*c^6+24*a^4*b^2*c^6+32*a^2*b^4*c^6-8*b^6*c^6-
21*a^4*c^8-38*a^2*b^2*c^8-5*b^4*c^8+16*a^2*c^10+12*b^2*c^10-5*c^12)*z^2 = 0
\end{verbatim}
}

\subsection*{Major vertices:}
Let $S$ be twice the area of the reference and:
{\scriptsize
\begin{verbatim}
rt = sqrt(a^6-3*a^2*b^4+2*b^6+6*a^2*b^2*c^2-2*b^4*c^2-3*a^2*c^4-2*b^2*c^4+2*c^6))
\end{verbatim}
}
The two major vertices of $\E_a^\dagger$ are given by:

{\scriptsize
\begin{verbatim}
[(a*(a^2-b^2-c^2)*((b^2-c^2)*(a^4*b^2-2*a^2*b^4+b^6+a^4*c^2+4*a^2*b^2*c^2-b^4*c^2-2*a^2*c^4-
b^2*c^4+c^6)+/-2*a^3*S*rt))/(a*(a^4*b^2-2*a^2*b^4+b^6+a^4*c^2+4*a^2*b^2*c^2-b^4*c^2-
2*a^2*c^4-b^2*c^4+c^6)+/-2*(b^2-c^2)*S*rt),b^2*(-a^2+b^2-c^2),-(c^2*(-a^2-b^2+c^2))]
\end{verbatim}
}
\subsection*{P-ellipse 6-point circle}

The center $X_3'$ of the 6-point circle of \cref{prop:anticevian} lies on the Van Aubel line ($X_4 X_6$) of the reference. It can be regarded as the circumcenter of the $X_3$-anticevian and is given by barycentrics $[f(a,b,c),f(b,c,a),f(c,a,b)]$ where $f(a,b,c)$ is given by:

{\scriptsize
\begin{verbatim}
(a^14-5*a^12*b^2+9*a^10*b^4-5*a^8*b^6-5*a^6*b^8+9*a^4*b^10-5*a^2*b^12+b^14-5*a^12*c^2+10*a^10*b^2*c^2-
13*a^8*b^4*c^2+28*a^6*b^6*c^2-31*a^4*b^8*c^2+10*a^2*b^10*c^2+b^12*c^2+9*a^10*c^4-13*a^8*b^2*c^4-
30*a^6*b^4*c^4+22*a^4*b^6*c^4+21*a^2*b^8*c^4-9*b^10*c^4-5*a^8*c^6+28*a^6*b^2*c^6+22*a^4*b^4*c^6-
52*a^2*b^6*c^6+7*b^8*c^6-5*a^6*c^8-31*a^4*b^2*c^8+21*a^2*b^4*c^8+7*b^6*c^8+9*a^4*c^10+10*a^2*b^2*c^10-
9*b^4*c^10-5*a^2*c^12+b^2*c^12+c^14)*a^2
\end{verbatim}
}

\subsection*{P-hyperbolas}

Let $La= |PB|-|PC|$, $Lb=|PC|-|PA|$, and $Lc=|PA|-|PB|$. Points on the $\H_a^*$ P-hyperbola satisfy:

{\scriptsize
\begin{verbatim}
(2*(b^2-c^2-La^2)*p*q+(a^2-La^2)*q^2-2*(b^2-c^2+La^2)*p*r-2*(a^2+La^2)*q*r+(a^2-La^2)*r^2)*x^2-
2*(b^2-c^2-La^2)*p^2*x*y-(a^2-La^2)*p^2*y^2+2*(b^2-c^2+La^2)*p^2*x*z+2*(a^2+La^2)*p^2*y*z-
(a^2-La^2)*p^2*z^2 = 0
\end{verbatim}}

The conic $\P^*$ throught the vertices of the 3 P-hyperbolas is given by:
{\scriptsize
\begin{verbatim}
x^2+y^2+z^2-(2*(a^2+La^2)*y*z)/(a^2-La^2)-(2*(b^2+Lb^2)*z*x)/(b^2-Lb^2)-(2*(c^2+
Lc^2)*x*y)/(c^2-Lc^2)=0
\end{verbatim}}

The first barycentric coordinate for its center $X_{5452}$ is given by:

{\scriptsize
\begin{verbatim}
(a^2*((-2*La^2)/(a^2-La^2)+(b^2+Lb^2)/(b^2-Lb^2)+(c^2+Lc^2)/(c^2-Lc^2)))/(a^2-La^2)
\end{verbatim} }
With the other two computed cyclically.


\bibliographystyle{maa}
\bibliography{999_refs_orig,999_refs_rgk}

\end{document}